\documentclass[12pt]{article}
\usepackage{amssymb}
\usepackage{amsmath}
\usepackage{amsthm}
\usepackage{color}
\usepackage{comment}
\usepackage{geometry}
\usepackage{graphicx}
\usepackage{tikz}

\DeclareFontFamily{OT1}{pzc}{}
\DeclareFontShape{OT1}{pzc}{m}{it}{<-> s * [1.10] pzcmi7t}{}
\DeclareMathAlphabet{\mathpzc}{OT1}{pzc}{m}{it}

\let\originalleft\left
\let\originalright\right
\renewcommand{\left}{\mathopen{}\mathclose\bgroup\originalleft}
\renewcommand{\right}{\aftergroup\egroup\originalright}

\newcommand*\circled[1]{\tikz[baseline=(char.base)]{\node[shape=circle,draw,inner sep=1.2pt](char){#1};}}

\begin{document}

\def\bM{{\bf M}}
\def\bO{{\bf 0}}
\def\bW{{\bf W}}
\def\cF{\mathcal{F}}
\def\co{\mathpzc{o}}
\def\cR{\mathcal{R}_{\rm BYG}}
\def\cS{\mathcal{S}}
\def\cT{\mathcal{T}}
\def\cW{\mathcal{W}}
\def\cX{\mathcal{X}}
\def\ee{\varepsilon}
\def\pM{\Phi}    
\def\rD{{\rm D}}
\def\re{{\rm e}}
\def\ri{{\rm i}}

\def\codeIndent{\hspace*{11mm}}

\definecolor{codeGreen}{rgb}{0,.5,0}
\definecolor{codeRed}{rgb}{.8,0,.1}
\definecolor{myPlotGrey}{rgb}{.85,.85,.85}
\definecolor{myLightRed}{rgb}{1,.65,.65}

\newcommand{\myStep}[2]{{\bf Step #1} --- #2}

\makeatletter
\newcommand{\overleftrightsmallarrow}{\mathpalette{\overarrowsmall@\leftrightarrowfill@}}
\newcommand{\overarrowsmall@}[3]{%
  \vbox{%
    \ialign{%
      ##\crcr
      #1{\smaller@style{#2}}\crcr
      \noalign{\nointerlineskip}%
      $\m@th\hfil#2#3\hfil$\crcr
    }%
  }%
}
\def\smaller@style#1{%
  \ifx#1\displaystyle\scriptstyle\else
    \ifx#1\textstyle\scriptstyle\else
      \scriptscriptstyle
    \fi
  \fi
}
\makeatother
\newcommand{\lineFull}[1]{\overleftrightsmallarrow{#1}}

\newcommand{\doroverline}[2]{\overline{#1#2}}
\newcommand{\roverline}[1]{\mathpalette\doroverline{#1}}
\newcommand{\lineSeg}[2]{\roverline{#1 #2}}

\newtheorem{theorem}{Theorem}[section]
\newtheorem{corollary}[theorem]{Corollary}
\newtheorem{lemma}[theorem]{Lemma}
\newtheorem{proposition}[theorem]{Proposition}
\newtheorem{algorithm}[theorem]{Algorithm}

\theoremstyle{definition}
\newtheorem{definition}{Definition}[section]
\newtheorem{example}[definition]{Example}

\theoremstyle{remark}
\newtheorem{remark}{Remark}[section]





\title{
Detecting invariant expanding cones for generating word sets to identify chaos in piecewise-linear maps.
}
\author{
D.J.W.~Simpson\\\\
School of Fundamental Sciences\\
Massey University\\
Palmerston North\\
New Zealand
}
\maketitle

\begin{abstract}
We show how the existence of three objects, $\Omega_{\rm trap}$, $\bW$, and $C$, for a continuous piecewise-linear map $f$ on $\mathbb{R}^N$, implies that $f$ has a topological attractor with a positive Lyapunov exponent. First, $\Omega_{\rm trap} \subset \mathbb{R}^N$ is trapping region for $f$. Second, $\bW$ is a finite set of words that encodes the forward orbits of all points in $\Omega_{\rm trap}$. Finally, $C \subset T \mathbb{R}^N$ is an invariant expanding cone for derivatives of compositions of $f$ formed by the words in $\bW$. We develop an algorithm that identifies these objects for two-dimensional homeomorphisms comprised of two affine pieces. The main effort is in the explicit construction of $\Omega_{\rm trap}$ and $C$. Their existence is equated to a set of computable conditions in a general way. This results in a computer-assisted proof of chaos throughout a relatively large regime of parameter space. We also observe how the failure of $C$ to be expanding can coincide with a bifurcation of $f$. Lyapunov exponents are evaluated using one-sided directional derivatives so that forward orbits that intersect a switching manifold (where $f$ is not differentiable) can be included in the analysis.
\end{abstract}

\section{Introduction}
\label{sec:intro}
\setcounter{equation}{0}

Piecewise-linear maps form canonical representations of nonlinear dynamics
and provide effective models of diverse physical systems \cite{DiBu08}.
Much work has been done to identify properties that imply a piecewise-linear map has a chaotic attractor.
These include Markov partitions \cite{GlWo11}, homoclinic connections \cite{Mi80},
and situations where the dynamics is effectively one-dimensional \cite{Ko05}.
In the context of ergodic theory, an attractor with an absolutely continuous invariant measure
exists for piecewise-expanding maps \cite{Bu99,Ts01} 
and piecewise-smooth maps with certain expansion properties \cite{Ry04,Yo85}. 
However, when applied to the two-dimensional border-collision normal form (2d BCNF)
such properties have only been verified over regions of parameter space
that are small in comparison to where numerical simulations suggest chaotic attractors actually exist \cite{Gl16e,Gl17}.
To address this issue we use invariant expanding cones to bound Lyapunov exponents.
We show that this appears to be a highly effective method
for identifying chaotic attractors in a formal way.

For a continuous map $f$ on $\mathbb{R}^N$,
the Lyapunov exponent $\lambda(x,v)$ characterises the asymptotic 
rate of separation of the forward orbits of arbitrarily close points $x$ and $x + \delta v$: 
\begin{equation}
\left\| f^n(x + \delta v) - f^n(x) \right\| \sim \re^{\lambda n} \| \delta v \|.
\label{eq:lambdaAsyDefn}
\end{equation}
A positive Lyapunov exponent for bounded orbits is a standard indicator of chaos \cite{Me07}.
Now suppose there exist $N \times N$ matrices $A_i$ such that
\begin{equation}
f^n(x + \delta v) - f^n(x) = \delta A_{n-1} \cdots A_1 A_0 v + \co(\delta),
\label{eq:forwardOrbitsDifference}
\end{equation}
where $\co(\delta)$ vanishes faster than $\delta$ as $\delta \to 0$.
These matrices exist if $f$ is $C^1$: each $A_i$ is the Jacobian matrix $\rD f$ evaluated at $f^i(x)$.
If $\bM$ denotes the set of all $A_i$
and there exists $c > 1$ and a cone $C \subset T \mathbb{R}^N$ with the property that
\begin{equation}
\text{$M v \in C$ and $\| M v \| \ge c \| v \|$, for all $v \in C$ and all $M \in \bM$},
\label{eq:expansionConditionIntro}
\end{equation}
then immediately we have $\lambda(x,v) > 0$ for any $v \in C$.
This idea is attributed to Alekseev \cite{Al68} (see \cite{BaDr15,Wo85})
and is useful for establishing splitting and hyperbolicity in smooth dynamical systems \cite{DaYo17,Ne04,StTa14}.

Condition \eqref{eq:expansionConditionIntro} is rather strong as
it requires $C$ to be invariant and expanding for every matrix in the possibly infinite set $\bM$.
But if $f$ is piecewise-linear then $\bM$ contains only as many matrices as pieces in the map.
For this reason invariant expanding cones are perfectly suited,
and perhaps under-utilised, for analysing piecewise-linear maps.
Invariant cones were central to Misiurewicz's strategy for
establishing hyperbolicity and transitivity in the Lozi map \cite{Mi80}.
More recently in \cite{GlSi19} an invariant expanding cone was constructed for the 2d BCNF 
to finally prove the widely held conjecture that a chaotic attractor exists
throughout a physically-important parameter regime $\cR$ that was first highlighted in \cite{BaYo98}.

In this paper we present a simple but powerful generalisation of the above approach
and use it to show that chaotic attractors persist beyond $\cR$,
in fact in some places right up to where there exist stable low-period solutions.
The idea is to let $\bM$ consist of certain products of the $A_i$, rather than of the $A_i$ themselves.
Each product $M \in \bM$ is characterised by a word $\cW$ connecting its constitute matrices to the pieces of $f$.
As long as the set of all such words $\bW$ generates the symbolic itinerary of the forward orbit of $x$,
then again \eqref{eq:expansionConditionIntro} implies $\lambda(x,v) > 0$.
This approach is quite flexible because there is a large amount of freedom in our choice of the set of words $\bW$.

To prove $f$ has a chaotic attractor
one needs to identify a trapping region $\Omega_{\rm trap} \subset \mathbb{R}^N$
(which ensures there exists a topological attractor),
a set of words $\bW$ that generates the symbolic itineraries for all $x \in \Omega_{\rm trap}$,
and an invariant expanding cone $C$ for the matrices corresponding to the words in $\bW$.
Below we find these objects for the 2d BCNF.
More precisely, we propose a way by which $\Omega_{\rm trap}$ and $C$ can be constructed for a particular word set $\bW$
and prove that all three objects have the required properties if a certain set of computable conditions are met.
While, for a given combination of parameter values, these conditions could be checked by hand,
it is more appropriate to check them numerically.
Below we formulate this as an algorithm (Algorithm \ref{al:theAlgorithm}).

The remainder of this paper is organised as follows.
We start in \S\ref{sec:bcnf} by showing where Algorithm \ref{al:theAlgorithm} detects a chaotic attractor
in a typical two-dimensional slice of the parameter space of the 2d BCNF.
Then in \S\ref{sec:theory} we clarify technical features mentioned above (cones, words, trapping regions, etc)
and express $\lambda(x,v)$ in terms of one-sided directional derivatives
in order to accommodate points whose forward orbits intersect a switching manifold (where $\rD f$ is not defined).
The result $\lambda(x,v) > 0$ is formalised by Theorem \ref{th:main}
for $N$-dimensional, continuous, piecewise-linear maps with two pieces.
The theorem is framed in terms of $\bW$-recurrent sets
(these are sets to which forward orbits return following one or more words in $\bW$).
Such sets provide a practical way by which the approach can be applied to concrete examples,
and in \S\ref{sec:mainProof} we show they imply that $\bW$ generates symbolic itineraries as needed.
Here we also characterise the matrices $A_i$ in the expression \eqref{eq:forwardOrbitsDifference}.
This is quite subtle because if $f^i(x)$ lies on a switching manifold then $A_i$ depends on $v$ as well as $x$.
Section \ref{sec:mainProof} concludes with a proof of Theorem \ref{th:main}.

In subsequent sections we work to apply this methodology to the 2d BCNF.
In \S\ref{sec:iec} we consider sets of $2 \times 2$ matrices
and devise a set of conditions for the existence of an invariant expanding cone.
In \S\ref{sec:fir} we connect consecutive points of an orbit to construct a polygon $\Omega$.
We then identify conditions implying $\Omega$ is forward invariant
and can be perturbed into a trapping region $\Omega_{\rm trap}$.
In \S\ref{sec:algorithm} we state Algorithm \ref{al:theAlgorithm} and prove its validity.
Here we comment further on the application of the algorithm to the 2d BCNF 
and discuss instances in which failure of the algorithm coincides with a bifurcation of $f$
at which the chaotic attractor is destroyed.
Concluding remarks are provided in \S\ref{sec:conc}.

\section{Chaotic attractors in the two-dimensional border-collision normal form}
\label{sec:bcnf}
\setcounter{equation}{0}

Let $f$ be a continuous map on $\mathbb{R}^2$ that is affine on each side of a line $\Sigma$.
Assume coordinates $x = (x_1,x_2) \in \mathbb{R}^2$ are chosen so that $\Sigma = \left\{ x \,\big|\, x_1 = 0 \right\}$.
If $f$ is generic in the sense that $f(\Sigma)$ intersects $\Sigma$ at exactly one point,
and this point is not a fixed point of $f$,
then there exists an affine coordinate change that puts $f$ into the form
\begin{equation}
f(x) = \begin{cases}
\begin{bmatrix} \tau_L & 1 \\ -\delta_L & 0 \end{bmatrix}
\begin{bmatrix} x_1 \\ x_2 \end{bmatrix} +
\begin{bmatrix} 1 \\ 0 \end{bmatrix}, & x_1 \le 0, \\
\begin{bmatrix} \tau_R & 1 \\ -\delta_R & 0 \end{bmatrix}
\begin{bmatrix} x_1 \\ x_2 \end{bmatrix} +
\begin{bmatrix} 1 \\ 0 \end{bmatrix}, & x_1 \ge 0,
\end{cases}
\label{eq:bcnf}
\end{equation}
where $\tau_L,\tau_R,\delta_L,\delta_R \in \mathbb{R}$.
The coordinate change required to arrive at \eqref{eq:bcnf} is provided in Appendix \ref{app:coordChange}.

The four parameter family \eqref{eq:bcnf} is the 2d BCNF of \cite{NuYo92} except the value of
the border-collision bifurcation parameter (usually denoted $\mu$) is fixed at $1$.
The family \eqref{eq:bcnf} has been studied extensively
to understand border-collision bifurcations (where a fixed point collides with a switching manifold)
arising in many applications, particularly vibrating mechanical systems with impacts or friction \cite{DiBu08,Si16}.
If $\tau_L = -\tau_R$ and $\delta_L = \delta_R$ then \eqref{eq:bcnf} reduces to the Lozi map \cite{Lo78}.

\begin{figure}[b!]
\begin{center}
\includegraphics[height=10cm]{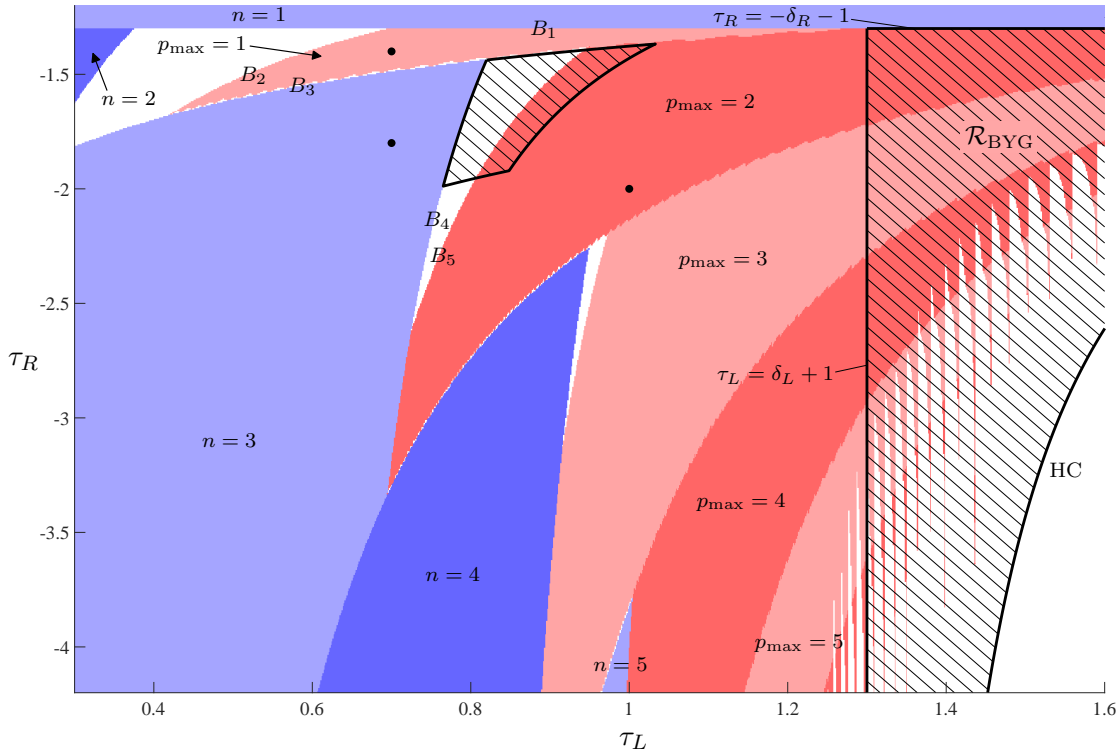}
\caption{
A two-parameter bifurcation diagram of the 2d BCNF \eqref{eq:bcnf} with \eqref{eq:dLdR}.
In the blue regions \eqref{eq:bcnf} has a stable period-$n$ solution for some $n \ge 1$
(in places where the $n=3$ and $n=4$ regions overlap only the $n=4$ region is shown).
In the red regions Algorithm \ref{al:theAlgorithm}
establishes the existence of chaotic attractor by using \eqref{eq:bW} for some $p_{\rm max} \ge 1$.
The three black dots indicate the parameter combinations examined in
Figs.~\ref{fig:RLpNumerics_ab}--\ref{fig:RLpNumerics_gh}
(see also Fig.~\ref{fig:RLpNumerics_ij}b).
The boundaries $B_1$ to $B_5$ are discussed in \S\ref{sec:algorithm}.
\label{fig:RLpNumerics_0}
}
\end{center}
\end{figure}

Fig.~\ref{fig:RLpNumerics_0} shows a two-dimensional slice
of the parameter space of \eqref{eq:bcnf} defined by fixing
\begin{align}
\delta_L &= 0.3, &
\delta_R &= 0.3.
\label{eq:dLdR}
\end{align}
Different values of $\delta_L, \delta_R \in (0,1)$ produce qualitatively similar pictures.
In the blue regions \eqref{eq:bcnf} has a stable period-$n$ solution for $1 \le n \le 5$.
To the right of the curve labelled HC (where the fixed point in $x_1 < 0$ attains a homoclinic connection) there is no attractor.
Numerical explorations suggest that in all other areas of Fig.~\ref{fig:RLpNumerics_0}
(i.e.~left of the HC curve and not inside a blue region) \eqref{eq:bcnf}
has a chaotic attractor \cite{BaGr99}.
These parameter values are physically relevant because \eqref{eq:bcnf}
is orientation-preserving and dissipative whenever $\delta_L, \delta_R \in (0,1)$.

The two striped regions of Fig.~\ref{fig:RLpNumerics_0} are of particular interest.
In the central striped region numerical simulations indicate that \eqref{eq:bcnf} has two chaotic attractors \cite{GlSi19b}.
The region $\cR$ (bounded by the HC curve and the lines $\tau_L = \delta_L + 1$ and $\tau_R = -\delta_R - 1$)
is the `robust chaos' parameter regime of \cite{BaYo98} (restricted to \eqref{eq:dLdR}).
This parameter regime is exactly where \eqref{eq:bcnf} has two saddle fixed points,
one with positive eigenvalues and one with negative eigenvalues.
Fig.~\ref{fig:RLpNumerics_ij}a shows a typical chaotic attractor in $\cR$.

\begin{figure}[b!]
\begin{center}
\includegraphics[height=4.9cm]{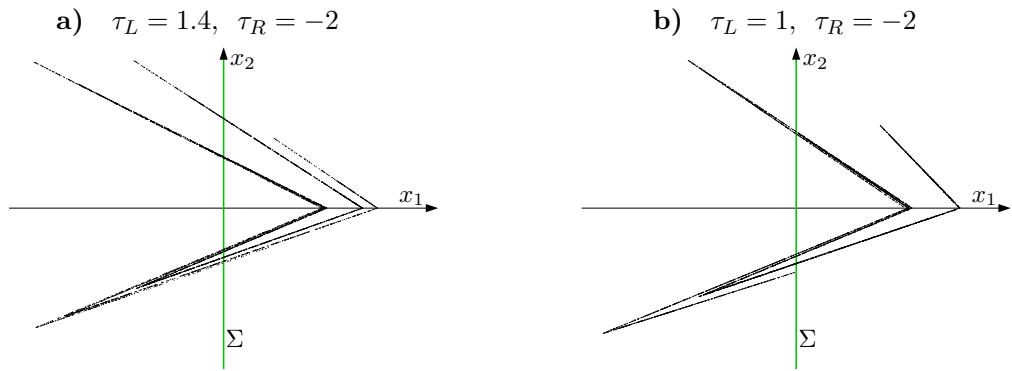}
\caption{
Chaotic attractors of \eqref{eq:bcnf} with \eqref{eq:dLdR}.
More precisely, these plots show $5000$ iterates of the forward orbit of the origin with some transient points removed.
\label{fig:RLpNumerics_ij}
}
\end{center}
\end{figure}

In \cite{GlSi19} it was shown that throughout $\cR$ \eqref{eq:bcnf}
has an attractor with a positive Lyapunov exponent.
This was achieved by constructing a trapping region and an invariant expanding cone for both matrices in \eqref{eq:bcnf}.
That is, the methodology of this paper was used with $\bW = \{ L, R \}$.
In \cite{GlSi19} it was not necessary to show that the symbolic itineraries of the forward orbits of points in the trapping region
are generated by $\bW$, because here $\bW$ generates all symbolic itineraries.

The approach of \cite{GlSi19} fails to find a chaotic attractor in $\tau_L < \delta_L + 1$
because here both eigenvalues of $A_L = \begin{bmatrix} \tau_L & 1 \\ -\delta_L & 0 \end{bmatrix}$ have modulus less than $1$,
so there does not exist an invariant expanding cone for $A_L$.
Therefore, if we are to show that \eqref{eq:bcnf} has a chaotic attractor with $\tau_L < \delta_L + 1$,
we cannot include $L$ in our word set $\bW$.
We instead use word sets of the form
\begin{equation}
\bW = \left\{ R L^p \,\big|\, 0 \le p \le p_{\rm max} \right\}.
\label{eq:bW}
\end{equation}
This is clarified in \S\ref{sub:words}.
The red regions of Fig.~\ref{fig:RLpNumerics_0} show where Algorithm \ref{al:theAlgorithm}
finds an attractor with a positive Lyapunov exponent by using \eqref{eq:bW}
with some $p_{\rm max} \ge 1$
over a $1024 \times 512$ grid of $(\tau_L,\tau_R)$-values
(see Fig.~\ref{fig:RLpNumerics_ij}b for a typical attractor).
The algorithm is highly effective in that it establishes chaos over about $90\%$
of parameter space between the blue regions and $\tau_L = \delta_L + 1$
(and also succeeds in part of $\cR$).
In fact in some places there is no gap between the blue and red regions
meaning that the algorithm finds a chaotic attractor right up to the bifurcation at which
the attractor is destroyed (this is discussed further in \S\ref{sec:algorithm}).

\section{Main definitions and a bound on the Lyapunov exponent}
\label{sec:theory}
\setcounter{equation}{0}

In this section we define the main objects
and state the main theoretical result, Theorem \ref{th:main}.
In order to arrive at Theorem \ref{th:main} quickly
some discussion is deferred to \S\ref{sec:mainProof}.

\subsection{Trapping regions and topological attractors}
\label{sub:attractors}

Here we provide topological definitions for a continuous map $f$ on $\mathbb{R}^N$,
see for instance \cite{GuHo86,Ro04} for further details.

\begin{definition}
A set $\Omega \subseteq \mathbb{R}^N$ is {\em forward invariant} if $f(\Omega) \subseteq \Omega$.
\label{df:forwardInvariant}
\end{definition}

\begin{definition}
A compact set $\Omega_{\rm trap} \subset \mathbb{R}^N$ is a {\em trapping region} if
$f \left( \Omega_{\rm trap} \right) \subseteq {\rm int} \left( \Omega_{\rm trap} \right)$.
(where ${\rm int}(\cdot)$ denotes interior).
\label{df:trappingRegion}
\end{definition}

\begin{definition}
An {\em attracting set} is
$\Lambda = \cap_{i=0}^\infty f^i(\Omega_{\rm trap})$
for some trapping region $\Omega_{\rm trap} \subset \mathbb{R}^N$.
\label{df:attractingSet}
\end{definition}

Topological attractors are invariant subsets of an attracting set that satisfy some kind of indivisibility condition
(e.g.~they contain a dense orbit) \cite{Me07}.
In this paper it is not necessary to consider such conditions as 
we only seek to show that $f$ has an attractor and this is achieved by showing that $f$ has a trapping region.

\subsection{Cones}
\label{sub:cones}

We denote the tangent space to a point $x \in \mathbb{R}^N$ by $T \mathbb{R}^N$.
The tangent space is isomorphic to $\mathbb{R}^N$ and
indeed we often treat tangent vectors $v \in T \mathbb{R}^N$ as elements of $\mathbb{R}^N$,
as in the expression $x + \delta v$.

\begin{definition}
A nonempty set $C \subseteq T \mathbb{R}^N$
is said be to a {\em cone}
if $t v \in C$ for all $v \in C$ and all $t \in \mathbb{R}$.
\label{df:cone}
\end{definition}

\begin{definition}
Let $M$ be an $N \times N$ matrix.
A cone $C \subseteq T \mathbb{R}^N$ is said to be {\em invariant} under $M$ if
\begin{equation}
M v \in C, \qquad \text{for all}~ v \in C.
\label{eq:coneInvariant}
\end{equation}
The cone is said to be {\em expanding} under $M$ if there exists $c > 1$ such that
\begin{equation}
\| M v \| \ge c \| v \|, \qquad \text{for all}~ v \in C,
\label{eq:coneExpanding}
\end{equation}
see Fig.~\ref{fig:RLpSchem_k}.
For a finite set of real-valued $N \times N$ matrices $\bM$, we say
$C$ is {\em invariant} [resp.~{\em expanding}]
if it is invariant [resp.~expanding] under every $M \in \bM$.
\label{df:iec}
\end{definition}

\begin{figure}[t!]
\begin{center}
\includegraphics[height=3cm]{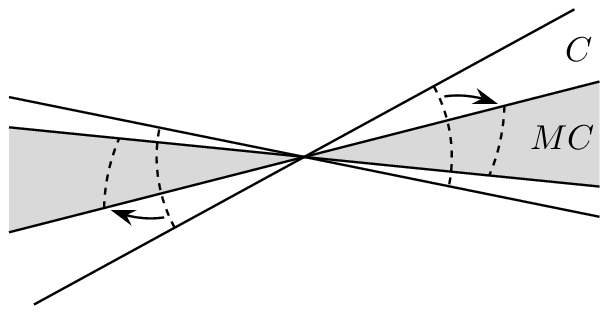}
\caption{
A sketch of an invariant expanding cone $C$
and its image $M C = \left\{ M v \,\big|\, v \in C \right\}$.
\label{fig:RLpSchem_k}
}
\end{center}
\end{figure}

\subsection{One-sided directional derivatives and Lyapunov exponents}
\label{sub:directionalDerivatives}

The {\em one-sided directional derivative} of a function $\phi : \mathbb{R}^N \to \mathbb{R}^N$
at $x \in \mathbb{R}^N$ in a direction $v \in T \mathbb{R}^N$ is defined as
\begin{equation}
\rD^+_v \phi(x) = \lim_{\delta \to 0^+} \frac{\phi(x + \delta v) - \phi(x)}{\delta},
\label{eq:Dplus}
\end{equation}
if this limit exists \cite{DhDu12,Ro70b}.
If $\rD^+_v f^n(x)$ exists for all $n \ge 1$, then taking $\delta \to 0^+$ in \eqref{eq:lambdaAsyDefn} gives
\begin{equation}
\left\| \rD^+_v f^n(x) \right\| \sim \re^{\lambda n} \| v \|.
\nonumber
\end{equation}
By further taking $n \to \infty$ we arrive at
\begin{equation}
\lambda(x,v) = \limsup_{n \to \infty} \frac{1}{n} \ln \left( \left\| \rD^+_v f^n (x) \right\| \right),
\label{eq:lambdaDefn}
\end{equation}
where, following usual convention, the supremum limit is taken
because from the point of view of ascertaining stability
one wants to record the largest possible fluctuations.

If $f$ is $C^1$ then the Jacobian matrix $\rD f$ exists everywhere and \eqref{eq:lambdaDefn} becomes
\begin{equation}
\lambda(x,v) = \limsup_{n \to \infty} \frac{1}{n} \ln \left( \left\| \rD f^n (x) v \right\| \right).
\nonumber
\end{equation}
Oseledet's theorem gives conditions under which $\lambda(x,v)$ takes at most $N$ values
for almost all $x$ in an invariant set \cite{BaPe13,Vi14}.
But this is often not the case for piecewise-linear maps.
As a minimal example, consider the one-dimensional map
\begin{equation}
f(x) = \begin{cases}
a_L x, & x \le 0, \\
a_R x, & x \ge 0,
\end{cases}
\label{eq:f1d}
\end{equation}
shown in Fig.~\ref{fig:RLpSchem_g}.
If $a_L, a_R > 0$ and $a_L \ne a_R$ then the fixed point $x=0$ has two different Lyapunov exponents:
$\lambda(0,-1) = \ln(a_L)$ and $\lambda(0,1) = \ln(a_R)$.

\begin{figure}[t!]
\begin{center}
\includegraphics[height=3cm]{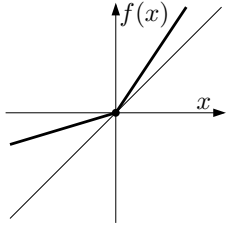}
\caption{
A sketch of the one-dimensional map \eqref{eq:f1d}.
\label{fig:RLpSchem_g}
}
\end{center}
\end{figure}

\subsection{Two-piece, piecewise-linear, continuous maps}
\label{sub:f}

For ease of explanation we develop our methodology for piecewise-linear maps comprised of only two pieces.
The extension to maps with more pieces is expected to be reasonably straight-forward.
Specifically we consider maps of the form
\begin{equation}
f(x) = \begin{cases}
A_L x + b, & x_1 \le 0, \\
A_R x + b, & x_1 \ge 0,
\end{cases}
\label{eq:f}
\end{equation}
where $A_L$ and $A_R$ are $N \times N$ matrices and $b \in \mathbb{R}^N$.
The assumption that $f$ is continuous on the switching manifold $\Sigma = \left\{ x \,\big|\, x_1 = 0 \right\}$
implies that $A_L$ and $A_R$ differ in only their first columns.
That is,
\begin{equation}
A_R - A_L = \zeta e_1^{\sf T},
\label{eq:ARminusAL}
\end{equation}
for some $\zeta \in \mathbb{R}^N$,
where $e_1$ is the first standard basis vector of $\mathbb{R}^N$.

If $\det(A_L) \ne 0$ and $\det(A_R) \ne 0$ then $f$ maps
the half-spaces $x_1 \le 0$ and $x_1 \ge 0$ in a one-to-one fashion to half-spaces with boundary $f(\Sigma)$.
If $\det(A_L) \det(A_R) < 0$ then $f$ is not invertible (it is of type $Z_0$-$Z_2$ \cite{MiGa96})
because $x_1 \le 0$ and $x_1 \ge 0$ are mapped to the same half-space.
If $\det(A_L) \det(A_R) > 0$ then $f$ is invertible (see \cite{Si16} for an explicit expression for $f^{-1}$)
and so we have the following result.

\begin{lemma}
The map \eqref{eq:f} is invertible if and only if $\det(A_L) \det(A_R) > 0$.
\label{le:fInverse}
\end{lemma}

\subsection{Words as symbolic representations of finite parts of orbits}
\label{sub:words}

To describe the symbolic itineraries of orbits of \eqref{eq:f} relative to $\Sigma$
we use words (defined here) and symbol sequences (defined in \S\ref{sub:symbolSequences})
on the alphabet $\{ L, R \}$.

\begin{definition}
A {\em word} of length $n \ge 1$ is a function $\cW : \{ 0,1,\ldots,n-1 \} \to \{ L, R \}$
and we write $\cW = \cW_0 \cW_1 \cdots \cW_{n-1}$. 
\label{df:word}
\end{definition}

Sometimes we abbreviate $k$ consecutive instances of a symbol by putting $k$ as a superscript.
For example $\cW = RL^3 = RLLL$ is a word of length four with
$\cW_0 = R$, $\cW_1 = L$, $\cW_2 = L$, and $\cW_3 = L$.

In order to obtain words from orbits we first define the following set-valued function on $\mathbb{R}^N$:
\begin{equation}
\gamma_{\rm set}(x) = \begin{cases}
\{ L \}, & x_1 < 0, \\
\{ L, R \}, & x_1 = 0, \\
\{ R \}, & x_1 > 0.
\end{cases}
\label{eq:gammaSet}
\end{equation}
Then, given $x \in \mathbb{R}^N$ and $n \ge 1$, we define
\begin{equation}
\Gamma(x;n) = \left\{ \cW : \{ 0,1,\ldots,n-1 \} \to \{ L, R \} \,\Big|\,
\cW_i \in \gamma_{\rm set} \big( f^i(x) \big) ~\text{for all}~ i = 0,1,\ldots,n-1 \right\}.
\label{eq:Gamman}
\end{equation}
Notice that if $f^i(x) \in \Sigma$ for $m$ values of $i \in \{ 0,1,\ldots,n-1 \}$,
then $\Gamma(x;n)$ contains $2^m$ words.
For example for Fig.~\ref{fig:RLpSchem_m} we have $\Gamma(x;3) = \{ RLL, RRL \}$.

Symbolic representations are often instead defined in a way
that produces a unique word for every $x$ and $n$ \cite{HaZh98}. 
In \S\ref{sub:h} we will see how the above formulation is particularly convenient for describing $\rD^+_v f^n(x)$.

\begin{figure}[b!]
\begin{center}
\includegraphics[height=3cm]{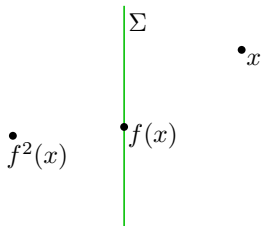}
\caption{
Part of an orbit for which $\Gamma(x;3) = \{ RLL, RRL \}$, see \eqref{eq:Gamman}.
\label{fig:RLpSchem_m}
}
\end{center}
\end{figure}

\subsection{Sufficient conditions for a positive Lyapunov exponent}
\label{sub:mainTheorem}

Here we state our main result for obtaining $\lambda(x,v) > 0$.
To do this we first provide a few more definitions
that are discussed further in \S\ref{sec:mainProof}.

Given a finite set of words $\bW$, let
\begin{equation}
\bM = \left\{ \pM(\cW) \,\middle|\, \cW \in \bW \right\},
\label{eq:bM}
\end{equation}
where
\begin{equation}
\pM(\cW) = A_{\cW_{n-1}} \cdots A_{\cW_1} A_{\cW_0} \,.
\label{eq:MW}
\end{equation}
Roughly speaking, $\pM(\cW)$ is equal to $\rD f^n$ for orbits
that map under $f$ following the word $\cW$.

\begin{definition}
A set $\Omega_{\rm rec} \subset \mathbb{R}^N$ is said to be {\em $\bW$-recurrent}
if for all $x \in \Omega_{\rm rec}$ there exists $n \ge 1$ such that
$f^n(x) \in \Omega_{\rm rec}$ and every $\cW \in \Gamma(x;n)$ can be written
as a concatenation of words in $\bW$.
\label{df:recurrent}
\end{definition}

\begin{theorem}
Suppose $\Omega_{\rm rec}$ is $\bW$-recurrent.
Suppose there exists an invariant expanding cone for $\bM$.
Then there exists $\lambda_{\rm bound} > 0$ such that for all $x \in \Omega_{\rm rec}$
\begin{equation}
\liminf_{n \to \infty} \frac{1}{n} \ln \left( \left\| \rD^+_v f^n(x) \right\| \right) \ge \lambda_{\rm bound} \,,
\qquad \text{for some}~v \in T \mathbb{R}^N.
\label{eq:liminf}
\end{equation}
Moreover, if $f$ is invertible and $\Omega_{\rm trap} \subset \bigcup_{i=-\infty}^\infty f^i(\Omega_{\rm rec})$ is a trapping region for $f$,
then $f$ has an attractor $\Lambda \subset \Omega_{\rm trap}$
and \eqref{eq:liminf} holds for all $x \in \Lambda$.
\label{th:main}
\end{theorem}

Observe \eqref{eq:liminf} implies $\lambda(x,v) > 0$
because the supremum limit is greater than or equal to the infimum limit.
Theorem \ref{th:main} is proved in the next section.

\section{Relationships between words, recurrent sets, and directional derivatives}
\label{sec:mainProof}
\setcounter{equation}{0}

Here we look deeper at the concepts introduced in the previous section for maps of the form \eqref{eq:f}.
First in \S\ref{sub:symbolSequences} we take the limit $n \to \infty$ in \eqref{eq:Gamman}
to obtain a set of symbol sequences for the forward orbit of a point $x \in \mathbb{R}^N$.
The key observation is that these sequences are generated by $\bW$ whenever $x$ belongs to a $\bW$-recurrent set.
Then in \S\ref{sub:h} we show that $D^+_v f^n(x)$ exists for all $x$, $v$, and $n$,
and describe it in terms of the product \eqref{eq:MW}.
Finally in \S\ref{sub:mainProof} we prove Theorem \ref{th:main}.

\subsection{Symbol sequences and generating word sets}
\label{sub:symbolSequences}

\begin{definition}
A {\em symbol sequence} is a function $\cS : \mathbb{N} \to \{ L, R \}$
and we write $\cS = \cS_0 \cS_1 \cS_2 \cdots$.
\label{df:symbolSequence}
\end{definition}

\begin{definition}
Let $\bW = \{ \cW^{(1)}, \ldots, \cW^{(k)} \}$ be a finite set of words.
We say that $\bW$ {\em generates} a symbol sequence $\cS$ if there exists
a sequence $\alpha : \mathbb{N} \to \{ 1,\ldots,k \}$ such that $\cS = \cW^{(\alpha_0)} \cW^{(\alpha_1)} \cW^{(\alpha_2)} \cdots$.
\label{df:generate}
\end{definition}

For example the set $\bW = \{ R, RL, RLL \}$ generates $\cS$ if and only if $\cS_0 = R$ and $\cS$ does not contain the word $LLL$.
In the language of coding $LLL$ is a {\em forbidden word}
and the set of sequences generated by $\bW$ is a {\em run-length limited shift} \cite{LiMa95}.

Next define
\begin{equation}
\Gamma(x) = \left\{ \cS : \mathbb{N} \to \{ L, R \} \,\Big|\,
\cS_i \in \gamma_{\rm set} \big( f^i(x) \big) ~\text{for all}~ i \ge 0 \right\},
\label{eq:Gamma}
\end{equation}
which represents the $n \to \infty$ limit of \eqref{eq:Gamman}.

\begin{lemma}
Let $\Omega_{\rm rec}$ be $\bW$-recurrent.
Then $\bW$ generates every $\cS \in \Gamma(x)$ for all $x \in \Omega_{\rm rec}$.
\label{le:generates}
\end{lemma}

\begin{proof}
Choose any $x \in \Omega_{\rm rec}$ and $\cS \in \Gamma(x)$.
Let $x_0 = x$ and $\cS^{(0)} = \cS$.
By an inductive argument we have that for all $j \ge 0$ there exists $n_j \ge 1$ such that
$x_{j+1} = f^{n_j}(x_j) \in \Omega_{\rm rec}$
and $\cS^{(j)} = \cX^{(j)} \cS^{(j+1)}$, where $\cS^{(j+1)} \in \Gamma(x_{j+1})$
and $\cX^{(j)}$ is a word of length $n_j$ that can be written as a concatenation of words in $\bW$.
Then $\cS = \cX^{(0)} \cX^{(1)} \cX^{(2)} \cdots$ as required.
\end{proof}

\subsection{A characterisation of one-sided directional derivatives}
\label{sub:h}

So far we have constructed sets of words $\Gamma(x;n)$ and sets of symbol sequences $\Gamma(x)$
to describe the forward orbit of a point $x$.
These have the advantage that they only depend on $x$, so they are relatively simple to describe and analyse.
However, in order to identify the matrices $A_i$ in the expression \eqref{eq:forwardOrbitsDifference}
(which we use to evaluate $D^+_v f^n(x)$),
we also require knowledge of $v$ (albeit only for forward orbits that intersect $\Sigma$).
To this end we define
\begin{equation}
\gamma((x,v)) = \begin{cases}
L \,, & x_1 < 0, ~\text{or}~ x_1 = 0 ~\text{and}~ v_1 < 0, \\
R \,, & x_1 > 0, ~\text{or}~ x_1 = 0 ~\text{and}~ v_1 \ge 0,
\end{cases}
\label{eq:gamma}
\end{equation}
where $(x,v)$ is an element of the tangent bundle $\mathbb{R}^N \times T \mathbb{R}^N$.
The choice of the symbol $R$ when $x_1 = v_1 = 0$
is not important to the results below because in this case $A_L v = A_R v$ by \eqref{eq:ARminusAL}.
We then have the following result (given also in \cite{Si20b}).

\begin{lemma}
For any $x \in \mathbb{R}^N$ and $v \in T \mathbb{R}^N$,
\begin{equation}
D^+_v f(x) = A_{\gamma((x,v))} v.
\label{eq:Dplus2}
\end{equation}
\label{le:Dplus}
\end{lemma}

\begin{proof}
If $x_1 > 0$ then $(x + \delta v)_1 > 0$ for sufficiently small values of $\delta > 0$.
In this case $f(x + \delta v) - f(x) = f_R(x + \delta v) - f_R(x) = \delta A_R v$, and so $D^+_v f(x) = A_R v$.
The same calculation occurs in the case $x_1 = 0$ and $v_1 \ge 0$
because we may take $f(x) = f_R(x)$ (by the continuity of $f$ on $\Sigma$)
and $(x + \delta v)_1 \ge 0$ so we may similarly take $f(x + \delta v) = f_R(x + \delta v)$.
The same arguments apply to $f^L$ if $x_1 < 0$, or $x_1 = 0$ and $v_1 < 0$, and result in $D^+_v f(x) = A_L v$.
\end{proof}

Higher directional derivatives are dictated by the evolution of tangent vectors.
For this reason we define the following map on $\mathbb{R}^N \times T \mathbb{R}^N$:
\begin{equation}
h((x,v)) = \left( f(x), D^+_v f(x) \right).
\label{eq:h}
\end{equation}
Since $D^+$ satisfies composition rule
\begin{equation}
D^+_v f^{i+j}(x) = D^+_{D^+_v f^i(x)} f^j \left( f^i(x) \right), \qquad \text{for any $i, j \ge 1$},
\label{eq:DplusComposition}
\end{equation}
the $n^{\rm th}$ iterate of $h$ is
\begin{equation}
h^n((x,v)) = \left( f^n(x), D^+_v f^n(x) \right).
\label{eq:hn}
\end{equation}
Then Lemma \ref{le:Dplus} implies
\begin{equation}
D^+_v f^n(x) = A_{\gamma \left( h^{n-1}((x,v)) \right)} \cdots A_{\gamma \left( h((x,v)) \right)} A_{\gamma((x,v))} v.
\nonumber
\end{equation}
Finally we can use \eqref{eq:MW} to write this as
\begin{equation}
D^+_v f^n(x) = \pM(\cW) v, \qquad \text{where}~ \cW_i = \gamma \left( h^i((x,v)) \right)
~\text{for each}~ i \in \{ 0,\ldots,n-1 \}.
\label{eq:Dplusfn}
\end{equation}
Equation \eqref{eq:Dplusfn} characterises the derivative $D^+_v f^n(x)$
and shows it exists for all $x \in \mathbb{R}^N$, $v \in T \mathbb{R}^N$, and $n \ge 1$.

\subsection{A lower bound on the Lyapunov exponent}
\label{sub:mainProof}

Here we work towards a proof of Theorem \ref{th:main}.
Note that in Lemma \ref{le:generatesBound} the bound on the Lyapunov exponent does not require that the cone is expanding,
but gives $\lambda_{\rm bound} > 0$ if $c > 1$.

\begin{lemma}
The map $h$ is invertible if and only if $f$ is invertible.
\label{le:hInverse}
\end{lemma}

\begin{proof}
The first component of $h$ is $f$.
If $\det(A_L) \det(A_R) \le 0$ then $f^{-1}$ is not well-defined by Lemma \ref{le:fInverse}
and thus $h^{-1}$ is also not well-defined.

Now suppose $\det(A_L) \det(A_R) > 0$.
By Lemma \ref{le:fInverse} the first component of $h$ is well-defined.
By Lemma \ref{le:Dplus} the second component of $h$ is $h_2(x,v) = A_{\gamma((x,v))} v$.
If $x_1 < 0$ then $h_2(x,v) = A_L v$ is invertible because $\det(A_L) \ne 0$.
Similarly if $x_1 > 0$ then $h_2(x,v) = A_R v$ is invertible because $\det(A_R) \ne 0$.
If $x_1 = 0$ then
\begin{equation}
h_2(x,v) = \begin{cases}
A_L v, & v_1 < 0, \\
A_R v, & v_1 \ge 0, \end{cases}
\nonumber
\end{equation}
which is equal to \eqref{eq:f} with $b = \bO$ (the zero vector) and so is invertible by Lemma \ref{le:fInverse}.
\end{proof}

\begin{lemma}
Let $\bW$ be a finite set of words and $\bM$ be given by \eqref{eq:bM}.
Suppose there exists an invariant cone $C$ satisfying \eqref{eq:coneExpanding} with some $c > 0$ for all $M \in \bM$.
Let $x \in \mathbb{R}^N$ and suppose $\bW$ generates every $\cS \in \Gamma(x)$.
Then \eqref{eq:liminf} is satisfied with $\lambda_{\rm bound} = \frac{\ln(c)}{L_{\rm max}}$,
where $L_{\rm max}$ is the length of the longest word(s) in $\bW$.
If $f$ is invertible the same bound applies to $f^i(x)$ for any $i \in \mathbb{Z}$.
\label{le:generatesBound}
\end{lemma}

\begin{proof}
Let $\cW^{(1)},\ldots,\cW^{(k)}$ be the words in $\bW$ and let $L_1,\ldots,L_k$ be their lengths.
Let $v \in C$ with $v \ne \bO$.
Define the sequence
\begin{equation}
\cS = \gamma((x,v)) \gamma \big( h((x,v)) \big) \gamma \left( h^2((x,v)) \right) \cdots.
\nonumber
\end{equation}
Then $\cS \in \Gamma(x)$ thus there exists a sequence $\alpha_j$ such that
\begin{equation}
\cS = \cW^{(\alpha_0)} \cW^{(\alpha_1)} \cW^{(\alpha_2)} \cdots.
\nonumber
\end{equation}
For each $j \ge 1$, let $n_j = L_{\alpha_0} + L_{\alpha_1} + \cdots + L_{\alpha_{j-1}}$ and let
\begin{equation}
v_j = \rD^+_v f^{n_j}(x).
\label{eq:vj}
\end{equation}
By \eqref{eq:Dplusfn} we have
$v_j = \pM \left( \cW^{(\alpha_0)} \cW^{(\alpha_1)} \cdots \cW^{(\alpha_{j-1})} \right) v$.
By \eqref{eq:MW} we have
\begin{equation}
\pM \left( \cW^{(\alpha_0)} \cW^{(\alpha_1)} \cdots \cW^{(\alpha_{j-1})} \right)
= \pM \left( \cW^{(\alpha_{j-1})} \right) \pM \left( \cW^{(\alpha_0)} \cW^{(\alpha_1)} \cdots \cW^{(\alpha_{j-2})} \right)
\nonumber
\end{equation}
and thus, for all $j \ge 1$,
\begin{equation}
v_j = \pM \left( \cW^{(\alpha_{j-1})} \right) v_{j-1} \,,
\label{eq:uj2}
\end{equation}
where $v_0 = v$.
Since $v_0 \in C$ and $C$ is invariant under each $\pM \left( \cW^{(\alpha_i)} \right)$,
we have $v_j \in C$ for all $j \ge 1$.
Moreover $\| v_j \| \ge c \| v_{j-1} \|$ and so $\| v_j \| \ge c^j \| v_0 \|$.
Then by \eqref{eq:vj}
\begin{align*}
\liminf_{n \to \infty} \frac{1}{n} \ln \left( \left\| \rD^+_v f^n(x) \right\| \right)
&= \liminf_{j \to \infty} \frac{1}{n_j} \ln \left( \| v_j \| \right) \\
&\ge \liminf_{j \to \infty} \frac{1}{j L_{\rm max}} \ln \left( c^j \| v \| \right) \\
&= \lambda_{\rm bound} \,.
\end{align*}

Finally suppose $f$ is invertible and choose any $i \in \mathbb{Z}$.
Write $(\tilde{x},\tilde{v}) = h^i(x,v)$.
In the case $i < 0$ this is well-defined by Lemma \ref{le:hInverse}.
By the composition rule \eqref{eq:DplusComposition} we have
\begin{equation}
\rD^+_{\tilde{v}} f^n(\tilde{x}) = \rD^+_v f^{n+i}(x),
\nonumber
\end{equation}
for any $n \ge 0$ for which $n+i \ge 0$.
Thus by taking $n \to \infty$ we obtain at the same bound for $\tilde{x} = f^i(x)$.
\end{proof}

\begin{proof}[Proof of Theorem \ref{th:main}]
Choose any $x \in \Omega_{\rm rec}$.
By Lemma \ref{le:generates}, $\bW$ generates every $\cS \in \Gamma(x)$.
Thus by Lemma \ref{le:generatesBound}, inequality \eqref{eq:liminf} holds for some $\lambda_{\rm bound} > 0$.
For the second part of the theorem,
an attractor $\Lambda \subset \Omega_{\rm trap}$ exists because $\Omega_{\rm trap}$ is a trapping region.
Choose any $y \in \Lambda$.
We can write $y = f^i(x)$ for some $i \in \mathbb{R}$,
thus \eqref{eq:liminf} holds for $y$ by the second part of Lemma \ref{le:generatesBound}.
\end{proof}

\section{Cones for sets of $2 \times 2$ matrices}
\label{sec:iec}
\setcounter{equation}{0}

For the remainder of the paper we apply the above methodology to maps on $\mathbb{R}^2$
with the Euclidean norm $\| \cdot \|$.

Let $M$ be a real-valued $2 \times 2$ matrix.
We are interested in the behaviour of $v \mapsto M v$,
where $v \in \mathbb{R}^2$, in regards to cones.
We start in \S\ref{sub:oneMatrix} by deriving properties of this map for vectors of the form
$v = (1,m)$, where $m \in \mathbb{R}$ is the slope of $v$.
The behaviour of scalar multiples of $(1,m)$ follows trivially by linearity.
The behaviour of $e_2 = (0,1)$ can be inferred by considering the limit $m \to \infty$.
However, this vector will not be of interest to us because if $\bW$ is given by \eqref{eq:bW}, then it contains $R$.
Notice $\| A_R e_2 \| = \| e_2 \|$ so
$e_2$ cannot belong to an invariant expanding cone for $\bM$ given by \eqref{eq:bM}.

In \S\ref{sub:manyMatrices} we consider several matrices $M^{(i)}$.
We use fixed points of $v \mapsto M^{(i)} v$ to construct a cone
and derive conditions sufficient for the cone to be invariant and expanding.

\subsection{Results for a single matrix $M$}
\label{sub:oneMatrix}

Write $v = (1,m)$ and
\begin{equation}
M = \begin{bmatrix} a & b \\ c & d \end{bmatrix}.
\label{eq:M}
\end{equation}
We first decompose $v \mapsto M v$ into two real-valued functions $G$ and $H$.
Let 
\begin{equation}
H(m) = \| M v \|^2 - \| v \|^2.
\label{eq:H2}
\end{equation}
This function is particularly amenable to analysis because it is quadratic in $m$:
\begin{equation}
H(m) = \left( b^2 + d^2 - 1 \right) m^2 + 2 (a b + c d) m + a^2 + c^2 - 1.
\label{eq:H}
\end{equation}
The factor $\frac{\| M v \|}{\| v \|}$ by which the norm of $v$ changes under multiplication by $M$
is less than $1$ if $H(m) < 0$, and greater than $1$ if $H(m) > 0$.

The slope of $M v$ is
\begin{equation}
G(m) = \frac{c + d m}{a + b m},
\label{eq:G}
\end{equation}
assuming $a + b m \ne 0$.
That is, $M v$ is a scalar multiple of $(1,G(m))$.
From \eqref{eq:G},
\begin{equation}
\frac{d G}{d m} = \frac{\det(M)}{(a + b m)^2},
\label{eq:dGdm}
\end{equation}
and so if $\det(M) > 0$
(which is the case below where $M$ is a product instances of $A_L$ and $A_R$ with $\det(A_L) > 0$ and $\det(A_R) > 0$)
then $G$ is increasing on any interval for which $a + b m \ne 0$.
Fixed points of $G$ satisfy
\begin{equation}
b m^2 + (a - d) m - c = 0.
\label{eq:GfpEqn}
\end{equation}
Note that $m \in \mathbb{R}$ is a fixed point of $G$ if and only if
$v = (1,m)$ is an eigenvector of $M$.

\begin{lemma}	
Suppose
\begin{equation}
0 < \det(M) < \tfrac{1}{4} \,{\rm trace}(M)^2 ~\text{and}~ b \ne 0.
\label{eq:Massumptions}
\end{equation}
Then $G$ has exactly two fixed points.
At one fixed point, call it $m_{\rm stab}$, we have $\frac{d G}{d m} = \eta$, for some $0 < \eta < 1$,
and at the other fixed point, call it $m_{\rm unstab}$, we have $\frac{d G}{d m} = \frac{1}{\eta}$.
Moreover, $m_{\rm unstab}$ lies between $m_{\rm stab}$ and $m_{\text{\em blow-up}} = -\frac{a}{b}$
(see for example Fig.~\ref{fig:RLpSchem_e}).
\label{le:Gfps}
\end{lemma}

\begin{figure}[b!]
\begin{center}
\includegraphics[height=4.5cm]{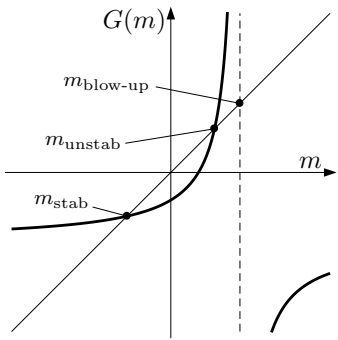}
\caption{
A sketch of the slope map \eqref{eq:G} when the stable fixed point $m_{\rm stab}$
has a smaller value than the unstable fixed point $m_{\rm unstab}$.
\label{fig:RLpSchem_e}
}
\end{center}
\end{figure}

\begin{proof}
By \eqref{eq:Massumptions} $M$ has distinct eigenvalues
$\lambda_1, \lambda_2 \in \mathbb{R}$ with $\lambda_1 \lambda_2 = \det(M) > 0$.
Without loss of generality assume $|\lambda_1| > |\lambda_2|$.
It is a simple exercise to show that $m_{\rm stab} = \frac{\lambda_1 - a}{b}$ and $m_{\rm unstab} = \frac{\lambda_2 - a}{b}$
satisfy the fixed point equation \eqref{eq:GfpEqn}.
These are only fixed points of $G(m)$ because \eqref{eq:GfpEqn} is quadratic.
By evaluating \eqref{eq:dGdm} at $m = m_{\rm stab}$ we obtain $\frac{d G}{d m} = \frac{\det(M)}{\lambda_1^2} = \frac{\lambda_2}{\lambda_1}$.
So $\eta = \frac{\lambda_2}{\lambda_1}$ and indeed $\eta \in (0,1)$.
Similarly at $m = m_2$ we have $\frac{d G}{d m} = \frac{\lambda_1}{\lambda_2} = \frac{1}{\eta}$.

Now suppose $m_{\rm stab} < m_{\text{blow-up}}$.
By the intermediate value theorem, $G(m)$ has a fixed point between $m_{\rm stab}$ and $m_{\text{blow-up}}$
because $\frac{d G}{d m} < 1$ at $m_{\rm stab}$ whereas $G(m) \to \infty$ as $m$ converges to $m_{\text{blow-up}}$ from the right.
This fixed point must be $m_{\rm unstab}$, thus we have $m_{\rm stab} < m_{\rm unstab} < m_{\text{blow-up}}$.
If instead $m_{\rm stab} > m_{\text{blow-up}}$,
an analogous argument produces $m_{\text{blow-up}} < m_{\rm unstab} < m_{\rm stab}$.
\end{proof}

\subsection{Results for a set of matrices $\bM$}
\label{sub:manyMatrices}

Let $\bM = \left\{ M^{(1)}, \ldots, M^{(k)} \right\}$ be a set of real-valued $2 \times 2$ matrices.
Write
\begin{equation}
M^{(j)} = \begin{bmatrix} a_j & b_j \\ c_j & d_j \end{bmatrix},
\label{eq:Mj}
\end{equation}
for each $j \in \{ 1,\ldots,k \}$, and let
\begin{align}
G_j(m) &= \frac{c_j + d_j m}{a_j + b_j m}, \label{eq:Gj} \\
H_j(m) &= \left( b_j^2 + d_j^2 - 1 \right) m^2 + 2 (a_j b_j + c_j d_j) m + a_j^2 + c_j^2 - 1, \label{eq:Hj}
\end{align}
denote \eqref{eq:H} and \eqref{eq:G} applied to \eqref{eq:Mj}.
Assume
\begin{equation}
0 < \det \left( M^{(j)} \right) < \tfrac{1}{4} \,{\rm trace} \left( M^{(j)} \right)^2 ~\text{and}~ b_j \ne 0,
\label{eq:Mjassumptions}
\end{equation}
for each $j$ so that each $M^{(j)}$ satisfies the assumptions of Lemma \ref{le:Gfps}.
Let $m^{(j)}_{\rm stab}$ and $m^{(j)}_{\rm unstab}$ denote the stable and unstable fixed points of \eqref{eq:Gj}
and let $m_{\rm stab,min}$ and $m_{\rm stab,max}$ denote the minimum and maximum values of
$\left\{ m^{(1)}_{\rm stab}, \ldots, m^{(k)}_{\rm stab} \right\}$.
Define the interval $J = \left[ m_{\rm stab,min}, m_{\rm stab,max} \right]$, see Fig.~\ref{fig:RLpSchem_f},
and the cone
\begin{equation}
C_J = \left\{ t \,(1,m) \,\big|\, t \in \mathbb{R}, m \in J \right\}.
\label{eq:CJ}
\end{equation}

\begin{figure}[b!]
\begin{center}
\includegraphics[height=4.5cm]{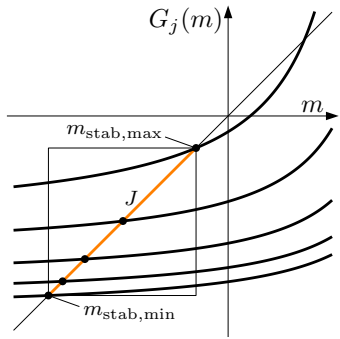}
\caption{
A sketch of five slope maps \eqref{eq:Gj}.
Here condition \eqref{eq:CIassumption1} is satisfied,
thus $G_j(J) \subset J$, for each $j$,
and so the cone $C_J$ \eqref{eq:CJ} is invariant under $\bM$, by Proposition \ref{pr:CJ}.
\label{fig:RLpSchem_f}
}
\end{center}
\end{figure}

\begin{proposition}
If
\begin{equation}
m^{(j)}_{\rm unstab} \notin J, \qquad \text{for each}~ j \in \{ 1,\ldots,k \},
\label{eq:CIassumption1}
\end{equation}
then $C_J$ is invariant under $\bM$.
If also
\begin{equation}
H_j(m) > 0, \qquad \text{for each}~ j \in \{ 1,\ldots,k \} ~\text{and all}~ m \in J,
\label{eq:CIassumption2}
\end{equation}
then $C_J$ is expanding under $\bM$.
\label{pr:CJ}
\end{proposition}

\begin{proof}
Choose any $j \in \{ 1,\ldots,k \}$.
Let $m^{(j)}_{\text{blow-up}}$ be the value of $m$ at which \eqref{eq:Gj} is undefined.
By assumption $m^{(j)}_{\rm unstab} \notin J$,
so by Lemma \ref{le:Gfps} if $m^{(j)}_{\rm unstab} < m_{\rm stab,min}$,
then $m^{(j)}_{\text{blow-up}} < m^{(j)}_{\rm unstab}$,
while if $m^{(j)}_{\rm unstab} > m_{\rm stab,max}$,
then $m^{(j)}_{\text{blow-up}} > m^{(j)}_{\rm unstab}$.
In either case $m^{(j)}_{\text{blow-up}} \notin J$, thus $G_j \big|_J$
(the restriction of $G_j$ to $J$) is continuous.

We now show that
\begin{equation}
\min_{m \in J} G_j(m) \ge m_{\rm stab,min} \,.
\label{eq:Gjmin}
\end{equation}
By \eqref{eq:dGdm} $G_j \big|_J$ is increasing because $\det \left( M^{(j)} \right) > 0$.
Thus $G_j \big|_J$ achieves its minimum at $m = m_{\rm stab,min}$.
If $m^{(j)}_{\rm stab} \ne m_{\rm stab,min}$ (otherwise \eqref{eq:Gjmin} is trivial)
then, since $\frac{d G_j}{d m} \big|_{m = m^{(j)}_{\rm stab}} < 1$ by Lemma \ref{le:Gfps},
we have $G_j(m) > m$ for some values of $m < m^{(j)}_{\rm stab}$ close to $m^{(j)}_{\rm stab}$.
Thus $G_j \left( m_{\rm stab,min} \right) > m_{\rm stab,min}$, for otherwise
$G_j$ would have another fixed point in $J$ by the immediate value theorem
and this is not possible because $m^{(j)}_{\rm stab}$ is unique fixed point of $G_j \big|_J$.
This verifies \eqref{eq:Gjmin}.
We similarly have $\max_{m \in J} G_j(m) \le m_{\rm stab,max}$.
Thus $G_j(m) \in J$ for all $m \in J$.
Thus $M^{(j)} u \in C_J$ for all $u \in C_J$.
Since $j$ is arbitrary, $C_J$ is forward invariant under $\bM$.

Finally, with $v = (1,m)$ we have
\begin{equation}
\min_{u \in C_J \setminus \{ \bO \}} \frac{\left\| M^{(j)} u \right\|}{\| u \|}
= \min_{m \in J} \frac{\left\| M^{(j)} v \right\|}{\| v \|}
= \min_{m \in J} \sqrt{\frac{H_j(m)}{\| v \|^2} + 1}
= c_j \,,
\nonumber
\end{equation}
where the minimum value $c_j$ is achieved because $J$ is compact.
If \eqref{eq:CIassumption2} is satisfied then $c_j > 1$.
Thus with ${\displaystyle c = \min_j c_j}$, $C_J$ is expanding under $\bM$.
\end{proof}

We complete this section by showing how the expansion condition \eqref{eq:CIassumption2}
can be checked with a finite set of calculations
based on the fact that each $H_j$ is quadratic in $m$.
If $H_j$ has two distinct real roots, then $H_j$ is increasing at one root, call it $m^{(j)}_{\rm incr}$,
and decreasing at the other root, call it $m^{(j)}_{\rm decr}$.

\begin{proposition}
Suppose $H_j$ has two distinct real roots.
Then
\begin{equation}
H_j(m) > 0, \qquad \text{for all}~ m \in J,
\label{eq:HjPositiveOnJ}
\end{equation}
if and only if two of the following inequalities are satisfied
\begin{align}
m_{\rm stab,max} &< m^{(j)}_{\rm decr} \,, \label{eq:inequality1} \\
m^{(j)}_{\rm decr} &< m^{(j)}_{\rm incr} \,, \label{eq:inequality2} \\
m^{(j)}_{\rm incr} &< m_{\rm stab,min} \,. \label{eq:inequality3}
\end{align}
\label{pr:threeInequalities}
\end{proposition}

Notice \eqref{eq:inequality1}--\eqref{eq:inequality3} cannot all be satisfied
because $m_{\rm stab,min} < m_{\rm stab,max}$.

\begin{proof}
The proof is achieved by brute-force; there are six cases to consider.

If \eqref{eq:inequality1} and \eqref{eq:inequality2} are true,
then $H_j(m) > 0$ for all $m < m^{(j)}_{\rm decr}$, which includes all $m \in J$,
thus \eqref{eq:HjPositiveOnJ} is true.
If \eqref{eq:inequality1} and \eqref{eq:inequality3} are true,
then $H_j(m) > 0$ for all $m \in \left( m^{(j)}_{\rm incr}, m^{(j)}_{\rm decr} \right)$ so \eqref{eq:HjPositiveOnJ} is true.
If \eqref{eq:inequality2} and \eqref{eq:inequality3} are true,
then $H_j(m) > 0$ for all $m > m^{(j)}_{\rm incr}$ so \eqref{eq:HjPositiveOnJ} is true.

If \eqref{eq:inequality1} and \eqref{eq:inequality2} are false,
then $H_j(m) < 0$ at $m = m_{\rm stab,max}$ so \eqref{eq:HjPositiveOnJ} is false.
If \eqref{eq:inequality1} and \eqref{eq:inequality3} are false
then $H_j(m) \le 0$ at $m = \min \!\big[ m^{(j)}_{\rm incr}, m_{\rm stab,max} \big]$, which belongs to $J$,
thus \eqref{eq:HjPositiveOnJ} is false.
If \eqref{eq:inequality2} and \eqref{eq:inequality3} are false,
then $H_j(m) < 0$ at $m = m_{\rm stab,min}$ so \eqref{eq:HjPositiveOnJ} is false.
\end{proof}

\begin{remark}
In the special case that $b_j^2 + d_j^2 - 1 = 0$ and $a_j b_j + c_j d_j \ne 0$,
the function $H_j$ has exactly one root:
\begin{equation}
m^{(j)}_{\rm root} = -\frac{a_j^2 + c_j^2 - 1}{2 (a_j b_j + c_j d_j)}.
\label{eq:mRoot}
\end{equation}
If $a_j b_j + c_j d_j < 0$ [resp.~$a_j b_j + c_j d_j > 0$]
then $H_j(m) > 0$ for all $m \in J$ if and only if $m_{\rm root} > m_{\rm stab,max}$
[resp.~$m_{\rm root} < m_{\rm stab,min}$].
This case arises in the implementation below because with $M^{(1)} = A_R$ we have $b_1 = 1$ and $d_1 = 0$.
\label{re:oneRoot}
\end{remark}

\section{The construction of a trapping region}
\label{sec:fir}
\setcounter{equation}{0}

In this section we study the 2d BCNF \eqref{eq:bcnf} in the orientation-preserving case: $\delta_L > 0$ and $\delta_R > 0$.
In this case the sign of $f(x)_2$ is opposite to the sign of $x_1$.
This implies points map between the quadrants of $\mathbb{R}^2$ as shown in Fig.~\ref{fig:RLpSchem_j}.
For example if $x$ belongs to the first quadrant, then $f(x)$ belongs to either the third quadrant or the fourth quadrant.

\begin{figure}[b!]
\begin{center}
\includegraphics[height=4.5cm]{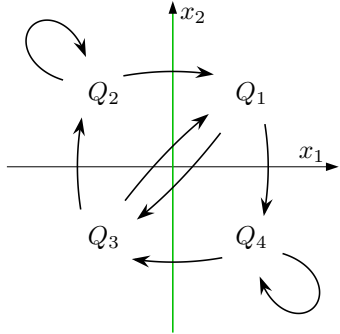}
\caption{
The action of the map \eqref{eq:bcnf} between the quadrants of $\mathbb{R}^2$
in the orientation-preserving setting, $\delta_L > 0$ and $\delta_R > 0$.
Here $Q_i$ denotes the closure of the $i^{\rm th}$ quadrant of $\mathbb{R}^2$.
For example if $x \in Q_1$ then $f(x) \in Q_3 \cup Q_4$.
\label{fig:RLpSchem_j}
}
\end{center}
\end{figure}

Let
\begin{align}
\Pi_L &= \left\{ x \in \mathbb{R}^2 \,\big|\, x_1 \le 0 \right\}, \nonumber \\
\Pi_R &= \left\{ x \in \mathbb{R}^2 \,\big|\, x_1 \ge 0 \right\}, \nonumber
\end{align}
denote the closed left and right half-planes.
Also write
\begin{align}
f_L(x) &= A_L x + b, \nonumber \\
f_R(x) &= A_R x + b, \nonumber
\end{align}
for the left and right half-maps of \eqref{eq:bcnf}.

\subsection{Preimages of the switching manifold under $f_L$}

With $\delta_L \ne 0$ the set $f_L^{-p}(\Sigma)$ is a line for all $p \ge 1$.
Below we use these lines to partition $\Pi_L$
into regions $D_p$ whose points escape $\Pi_L$ after exactly $p$ iterations of $f$,
see already Fig.~\ref{fig:RLpSchem_b}.

If $f_L^{-p}(\Sigma)$ is not vertical,
let $m_p$ denote its slope and let $c_p$ denote its $x_2$-intercept with $\Sigma$.
That is,
\begin{equation}
f_L^{-p}(\Sigma) = \left\{ x \in \mathbb{R}^2 \,\big|\, x_2 = m_p x_1 + c_p \right\}.
\label{eq:fLmpSigma}
\end{equation}
It is a simple 	exercise to show that
$m_1 = -\tau_L$, $c_1 = -1$, and for all $p \ge 1$,
\begin{align}
m_{p+1} &= -\frac{\delta_L}{m_p} - \tau_L \,, \label{eq:mMap} \\
c_{p+1} &= -\frac{c_p}{m_p} - 1, \label{eq:cMap}
\end{align}
whenever $m_p \ne 0$.
Fig.~\ref{fig:RLpSchem_a} shows a typical plot of \eqref{eq:mMap}.

\begin{figure}[b!]
\begin{center}
\includegraphics[height=4.5cm]{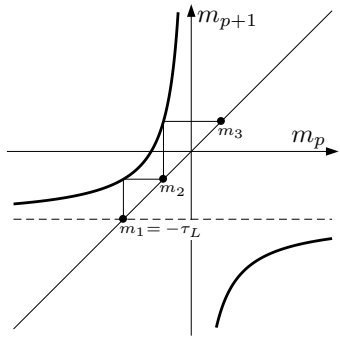}
\caption{
A typical plot of \eqref{eq:mMap}.
Here $p^* = 3$ (see Definition \ref{df:pStar}).
\label{fig:RLpSchem_a}
}
\end{center}
\end{figure}

\begin{definition}
Let $p^*$ be the smallest $p \ge 1$ for which $m_p \ge 0$,
with $p^* = \infty$ if $m_p < 0$ for all $p \ge 1$.
\label{df:pStar}
\end{definition}

Next we establish monotonicity of $f_L^{-p}(\Sigma)$ in $\Pi_L$ and provide an explicit expression for $p^*$.

\begin{lemma}
The sequence $\{ m_p \}_{p=1}^{p^*}$ is increasing;
the sequence $\{ c_p \}_{p=1}^{p^*}$ is decreasing.
\label{le:mpcp}
\end{lemma}

\begin{proof}
Since $m_1 = -\tau_L$, if $\tau_L \le 0$ then $p^* = 1$ and the result is trivial.
Suppose $\tau_L > 0$ (for which $p^* \ge 2$).
Let $g(m) = -\frac{\delta_L}{m} - \tau_L$ denote the map \eqref{eq:mMap}.
Since $g(m_1) - m_1 = \frac{\delta_L}{\tau_L} > 0$
and $\frac{d g}{d m} = \frac{\delta_L}{m^2} > 0$ for all $m < 0$,
the forward orbit of $m_1$ under $g$ is increasing while $m < 0$.
That is, $\{ m_p \}_{p=1}^{p^*}$ is increasing.

We now prove $\{ c_p \}_{p=1}^{p^*}$ is decreasing by induction.
Observe $c_1 > c_2$ because $c_1 = -1$ and $c_2 = -\frac{1}{\tau_L} - 1$ (where $\tau_L > 0$).
Thus it remains to consider $p^* \ge 3$.
Suppose $c_p > c_{p+1}$ for some $p \in \{ 1,\ldots,p^*-2 \}$
(this is our induction hypothesis).
It remains to show $c_{p+1} > c_{p+2}$.
Rearranging \eqref{eq:cMap} produces
\begin{equation}
\frac{c_{p+1} + 1}{-c_p} = \frac{1}{m_p}.
\nonumber
\end{equation}
But $c_{p+1} < c_p < 0$ and $m_p < m_{p+1} < 0$ thus
\begin{equation}
\frac{c_{p+1} + 1}{-c_{p+1}} > \frac{1}{m_{p+1}},
\nonumber
\end{equation}
which is equivalent to $c_{p+1} > -\frac{c_{p+1}}{m_{p+1}} - 1 = c_{p+2}$.
\end{proof}

\begin{proposition}
The value $p^*$ of Definition \ref{df:pStar} is given by
\begin{equation}
p^* = \begin{cases}
1, & \tau_L \le 0, \\
\left\lceil \frac{\pi}{\phi} - 1 \right\rceil, & 0 < \tau_L < 2 \sqrt{\delta_L} \,, \\
\infty, & \tau_L \ge 2 \sqrt{\delta_L} \,,
\end{cases}
\label{eq:pMax}
\end{equation}
where $\phi = \cos^{-1} \left( \frac{\tau_L}{2 \sqrt{\delta_L}} \right)
\in \left( 0, \frac{\pi}{2} \right)$, see Fig.~\ref{fig:RLpSchem_d}.
\label{pr:pStar}
\end{proposition}

\begin{figure}[b!]
\begin{center}
\includegraphics[height=4.5cm]{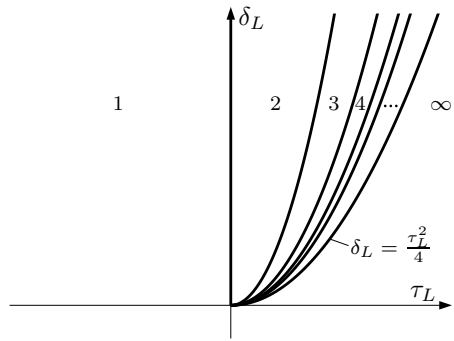}
\caption{
The value $p^*$ of Definition \ref{df:pStar} as given by Lemma \ref{le:mpcp}.
\label{fig:RLpSchem_d}
}
\end{center}
\end{figure}

\begin{proof}
Since $m_1 = -\tau_L$, if $\tau_L \le 0$ then $p^* = 1$.
If $\tau_L \ge 2 \sqrt{\delta_L}$ then
$m_\infty = \frac{1}{2} \left( -\tau_L - \sqrt{\tau_L^2 - 4 \delta_L} \right)$
is a fixed point of \eqref{eq:mMap}, call this map $g(m)$.
Moreover $m_1 < m_\infty < 0$ and $g(m) > m$ for all $m_1 \le m < m_\infty$
so $m_p \to m_\infty$ as $p \to \infty$.
Thus $p^* = \infty$ in this case.

If $0 < \tau_L < 2 \sqrt{\delta_L}$ then the eigenvalues of $A_L$ are
$\lambda_1 = \sqrt{\delta_L} \,\re^{\ri \phi}$
and $\lambda_2 = \sqrt{\delta_L} \,\re^{-\ri \phi}$.
Then $m_1 = -(\lambda_1 + \lambda_2)$ and \eqref{eq:mMap} can be written as
$m_{p+1} = -\frac{\lambda_1 \lambda_2}{m_p} - (\lambda_1 + \lambda_2)$.
It is a simple exercise to show from these (using induction on $p$) that while $\lambda_1^p \ne \lambda_2^p$ we have
\begin{equation}
m_p = \frac{\lambda_1^{p+1} - \lambda_2^{p+1}}{\lambda_2^p - \lambda_1^p},
\nonumber
\end{equation}
which we can rewrite as
\begin{equation}
m_p = -\frac{\sin((p+1) \phi)}{\sin(p \phi)} \sqrt{\delta_L} \,.
\label{eq:mpExplicit}
\end{equation}
By \eqref{eq:mpExplicit}, if $\frac{\pi}{p+1} \le \phi < \frac{\pi}{p}$,
then $m_p \ge 0$ and $m_j < 0$ for all $j \in \{ 1,\ldots,p-1 \}$, so $p^* = p$.
Since $\frac{\pi}{p+1} \le \phi < \frac{\pi}{p}$ is equivalent to
$p = \left\lceil \frac{\pi}{\phi} - 1 \right\rceil$, the proof is completed.
\end{proof}

\subsection{Partitioning the left half-plane by the number of iterations required to escape}

By Lemma \ref{le:mpcp}, in $\Pi_L$
each $f_L^{-p}(\Sigma)$ is located below $f_L^{-(p-1)}(\Sigma)$, for $p = 2,\ldots,p^*$, see Fig.~\ref{fig:RLpSchem_b}.
This implies that the regions $D_p \subset \Pi_L$, defined below, are disjoint.

\begin{figure}[b!]
\begin{center}
\includegraphics[height=6cm]{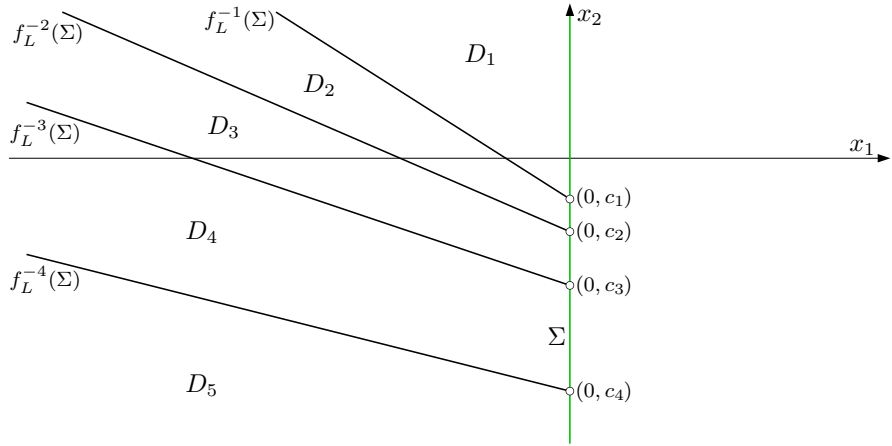}
\caption{
A typical plot of the first few preimages of the switching manifold $\Sigma$ under the left half-map $f_L$.
These lines are described by \eqref{eq:fLmpSigma} and bound the regions $D_p$ of Definition \ref{df:Dp}.
\label{fig:RLpSchem_b}
}
\end{center}
\end{figure}

\begin{definition}
Let
\begin{equation}
D_1 = \left\{ x \in \Pi_L \,\big|\, x_2 > m_1 x_1 + c_1 \right\}.
\label{eq:D1}
\end{equation}
For all finite $p \in \{ 2,\ldots,p^* \}$ let
\begin{equation}
D_p = \left\{ x \in \Pi_L \,\big|\, m_p x_1 + c_p < x_2 \le m_{p-1} x_1 + c_{p-1} \right\}.
\label{eq:Dp}
\end{equation}
If $p^* < \infty$ also let
\begin{equation}
D_{p^*+1} = \left\{ x \in \Pi_L \,\big|\, x_2 \le m_{p^*} x_1 + c_{p^*} \right\}.
\label{eq:DpStarp1}
\end{equation}
\label{df:Dp}
\end{definition}

Notice that if $p^* < \infty$ then $D_1,\ldots,D_{p^*+1}$ partition $\Pi_L$.
We now look the number of iterations of $f_L$ required for $x \in \Pi_L$ to escape $\Pi_L$.

\begin{definition}
Given $x \in \Pi_L$
let $\chi_L(x)$ be the smallest $p \ge 1$ for which $f_L^p(x) \notin \Pi_L$,
with $\chi_L(x) = \infty$ if there exists no such $p$.
\label{df:chiL}
\end{definition}

\begin{proposition}
Let $x \in \Pi_L$ and $p \in \{ 1,\ldots,p^*+1 \}$ be finite.
Then $x \in D_p$ if and only if $\chi_L(x) = p$.
\label{pr:partition}
\end{proposition}

\begin{proof}
We first show that $x \in D_p$ implies $\chi_L(x) = p$ for all $p \in \{ 1,\ldots,p^*+1 \}$ by induction on $p$.
If $x \in D_1$, then $x_2 > -\tau_L x_1 - 1$ (recalling $m_1 = -\tau_L$ and $c_1 = -1$)
and so $f_L(x)_1 = \tau_L x_1 + x_2 + 1 > 0$, hence $\chi_L(x) = 1$.
Suppose $x \in D_p$ implies $\chi_L(x) = p$ for some $p \in \{ 1,\ldots,p^* \}$ (this is our induction hypothesis).
Choose any $x \in D_{p+1}$.
Then $m_{p+1} x_1 + c_{p+1} < x_2 \le m_p x_1 + c_p$.
By using \eqref{eq:bcnf}, \eqref{eq:mMap}, and \eqref{eq:cMap}
we obtain $m_p f_L(x)_1 + c_p < f_L(x)_2 \le m_{p-1} f_L(x)_1 + c_{p-1}$,
except if $p = 1$ the latter inequality is absent.
Thus $f_L(x) \in D_p$ and so $\chi_L \left( f_L(x) \right) = p$ by the induction hypothesis.
Also $f_L(x)_1 \le 0$ (because $x \notin D_1$), therefore $\chi_L(x) = p+1$.

If $p^* < \infty$ the converse is true because $D_1,\ldots,D_{p^*+1}$ partition $\Pi_L$.
It remains show in the case $p^* = \infty$ that if $x \in E = \Pi_L \setminus \bigcup_{p=1}^\infty D_p$,
then $\chi_L(x) = \infty$.
We have
\begin{equation}
E = \left\{ x \in \Pi_L \,\big|\, x_2 \le m_\infty x_1 + c_\infty \right\},
\nonumber
\end{equation}
where $m_\infty = \lim_{p \to \infty} m_p$ is given in the proof of Proposition \ref{pr:pStar}
and $c_\infty = \lim_{p \to \infty} c_p = \frac{-m_\infty}{m_\infty + 1}$.
If $x \in E$ then $f_L(x)_2 \le m_\infty f_L(x)_1 + c_\infty$ (obtained by repeating the calculations in the above induction step)
and $f_L(x)_1 = \tau_L x_1 + x_2 + 1 \le c_\infty + 1$ (obtained by using also $x_1 \le 0$).
Thus $f_L(x)_1 < 0$ (because $c_\infty < -1$ by Lemma \ref{le:mpcp}).
Therefore $f_L(x) \in E$ which shows that $E$ is forward invariant under $f_L$.
Therefore $\chi_L(x) = \infty$ for any $x \in E$.
\end{proof}

\subsection{A forward invariant region and a trapping region}
\label{sub:Omega}

Let $\beta > 0$ and $X = (0,\beta)$.
Here we use the first few images and preimages of $X$ to form a polygon $\Omega$.
To do this we require the following assumption on $X$:
\begin{equation}
\text{There exist}~ i, j \ge 1 ~\text{such that}~
f^i(X) \in \Pi_L ~\text{and}~ f^{-j}(X) \in \Pi_R \,.
\label{eq:rlExist}
\end{equation}

\begin{definition}
Let $r$ and $\ell$ be the smallest values of $i$ and $j$ satisfying \eqref{eq:rlExist}, respectively.
\label{df:rl}
\end{definition}

That is, $f^i(X)_1 > 0$ for all $i = 1,\ldots,r-1$ and $f^r(X)_1 \le 0$.
Also $f^{-j}(X)_1 < 0$ for all $j = 1,\ldots,\ell-1$ and $f^{-\ell}(X)_1 \ge 0$.
Notice $r \ge 2$ because $f(X)_1 = \beta + 1 > 0$.
Also $\ell \ge 2$ because $f^{-1}(X)_1 = -\frac{\beta}{\delta_L} < 0$.

In what follows $\lineFull{P Q}$ denotes the line through distinct points $P, Q \in \mathbb{R}^2$.
Proofs of the next three results are deferred to the end of this section.

\begin{lemma}
Suppose \eqref{eq:rlExist} is satisfied.
Let $Z = f^r(X)$ and let $Y$ denote the intersection of $\lineFull{Z f^{-1}(Z)}$ with $\Sigma$.
Let $V = f^{-(\ell-1)}(X)$ and let $U$ denote the intersection of $\lineFull{V f(V)}$ with $f(\Sigma)$,
see Fig.~\ref{fig:RLpSchem_c}.
The closed polygonal chain formed by connecting the points
\begin{equation}
U, f^{-(\ell-2)}(X), \ldots, f^{-1}(X), X, f(X), \ldots, f^{r-1}(X), Z,
\label{eq:OmegaVertices}
\end{equation}
in order, and then from $Z$ back to $U$, has no self-intersections (it is a Jordan curve).
\label{le:simplePolygon}
\end{lemma}

\begin{figure}[b!]
\begin{center}
\includegraphics[height=7.5cm]{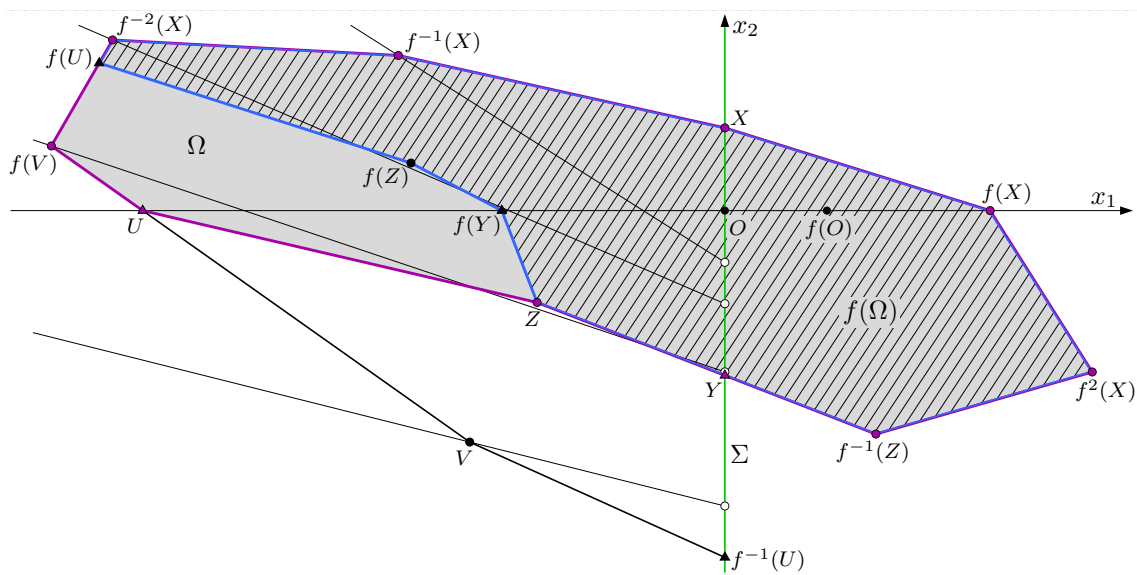}
\caption{
A plot of $\Omega$ (shaded) and $f(\Omega)$ (striped).
The values in Definition \ref{df:rl} are $r = 4$ (so $Z = f^4(X)$) and $\ell = 5$ (so $V = f^{-4}(X)$).
Here \eqref{eq:Yass}--\eqref{eq:Zass2} are satisfied
so $f(\Omega) \subset \Omega$ by Proposition \ref{pr:forwardInvariant}.
\label{fig:RLpSchem_c}
}
\end{center}
\end{figure}

\begin{definition}
Let $\Omega$ be the polygonal chain of Lemma \ref{le:simplePolygon} and its interior.
\label{df:Omega}
\end{definition}

Now we show that if three conditions are satisfied then $\Omega$ is forward invariant
and we can shrink it by an arbitrarily small amount to obtain a trapping region, $\Omega_{\rm trap}$.

\begin{proposition}
Suppose \eqref{eq:rlExist} is satisfied and
\begin{align}
&\text{$Y$ lies above $f^{-1}(U)$}, \label{eq:Yass} \\
&\text{$Z$ lies above $\lineFull{f^{-1}(U) V}$}, \label{eq:Zass1} \\
&\text{$Z$ lies to the right of $\lineFull{V f(V)}$}. \label{eq:Zass2}
\end{align}
Then $\Omega$ is forward invariant under $f$.
Moreover, for all $\ee > 0$ there exists a trapping region $\Omega_{\rm trap} \subset \Omega$ for $f$
with $d_H(\Omega,\Omega_{\rm trap}) \le \ee$, where $d_H$ denotes Hausdorff distance\footnote{
The {\em Hausdorff distance} between sets $\Omega_1$ and $\Omega_2$ is defined as
\begin{equation}
d_H(\Omega_1,\Omega_2) = \max \left[ \sup_{x \in \Omega_1} \inf_{y \in \Omega_2} \| x - y \|, \sup_{y \in \Omega_2} \inf_{x \in \Omega_1} \| x - y \| \right].
\nonumber
\end{equation}
}.
\label{pr:forwardInvariant}
\end{proposition}

The next result allows us to apply Theorem \ref{th:main}.
The actual application to Theorem \ref{th:main} is detailed in the next section.

\begin{proposition}
Suppose \eqref{eq:rlExist}, \eqref{eq:Yass}, \eqref{eq:Zass1}, and \eqref{eq:Zass2} are satisfied.
Let
\begin{equation}
\Omega_{\rm rec} = \left\{ x \in \Omega \,\big|\, x_1 > 0 \right\},
\label{eq:OmegaGen}
\end{equation}
and $\bW$ be given by \eqref{eq:bW} with
\begin{equation}
p_{\rm max} = \max \left[ \chi_L(Y), \chi_L(Z) \right].
\label{eq:pMax2}
\end{equation}
Then $\Omega_{\rm rec}$ is $\bW$-recurrent and $\Omega \subset \bigcup_{i=-\ell}^0 f^i(\Omega_{\rm rec})$.
\label{pr:goFromOmegaGen}
\end{proposition}

\begin{proof}[Proof of Lemma \ref{le:simplePolygon}]
Here we use the notation $[A,B)$ to denote the line segment $\{ (1-s) A + s B \,|\, 0 \le s < 1 \}$
for points $A, B \in \mathbb{R}^2$.

The points $X$ and $f(X)$ belong to $\Pi_R$, thus the line segment $[X, f(X))$, call it $\zeta$, is contained in $\Pi_R$.
The points $f^{-(\ell-1)}(X), \ldots, f^{-2}(X), f^{-1}(X)$ all belong to ${\rm int}(\Pi_L)$,
thus for each $i \in \{ 1,\ldots,\ell-1 \}$ the line segment $L_i = \big[ f^{-i}(X), f^{-(i-1)}(X) \big)$ is contained in ${\rm int}(\Pi_L)$.
These line segments are mutually disjoint because if $L_i$ and $L_j$, with $i < j$,
intersect at some point $P$,
then $P \in L_i$ implies $f^i(P) \in \zeta$ and so $f^i(P)_1 \ge 0$,
while $P \in L_j$ implies $f^i(P) \in L_{j-i}$ and so $f^i(P)_1 < 0$,
which is a contradiction.
By a similar argument the line segments $\big[ f^i(X), f^{i+1}(X) \big)$, for $i \in \{ 1,\ldots,r-1 \}$, are mutually disjoint
and contained in $x_2 \le 0$ (whereas $\zeta$ is contained in $x_2 > 0$).
This shows that the chain formed by connecting the points \eqref{eq:OmegaVertices} has no self-intersections.
The addition of $[Z, U)$ introduces no intersections
because $[Z, U)$ and part of $\left[ f_R^{r-1}(X), Z \right)$
are the only components of the chain that belong to the third quadrant.
\end{proof}

\begin{figure}[b!]
\begin{center}
\includegraphics[height=3cm]{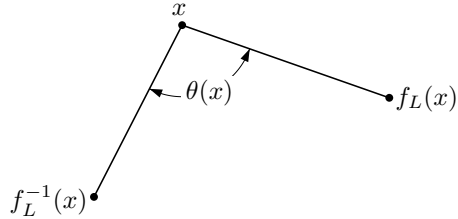}
\caption{
The angle $\theta(x)$ formed at a point $x$ by its preimage and image under $f_L$.
\label{fig:RLpSchem_o}
}
\end{center}
\end{figure}

\begin{proof}[Proof of Proposition \ref{pr:forwardInvariant}]~\\
\myStep{1}{Interior angles of $\Omega$.}\\
Define $\theta : \mathbb{R}^2 \to [0,2 \pi)$ as follows:
$\theta(x)$ is the angle at $x$ from $\left[ f_L^{-1}(x), x \right]$ anticlockwise to $\left[ x, f_L(x) \right]$,
see Fig.~\ref{fig:RLpSchem_o}.
Also define $g(x) = f_L(x) - x$ and
\begin{equation}
S(x) = g(x) \wedge g \left( f_L^{-1}(x) \right),
\label{eq:S}
\end{equation}
where we have defined the wedge product $A \wedge B = A_1 B_2 - A_2 B_1$.
It is a simple exercise to show that\footnote{
Equation \eqref{eq:sinTheta} is a version of the well-known formula $\sin(\theta) = \frac{\| u \times v \|}{\| u \| \| v \|}$
for the sine of the angle between $u, v \in \mathbb{R}^3$.
}
\begin{equation}
\sin(\theta(x)) = \frac{S(x)}{\| g(x) \| \left\| g \left( f_L^{-1}(x) \right) \right\|}.
\label{eq:sinTheta}
\end{equation}
From $f_L(x) = A_L x + b$ we obtain the formula $g(f_L(x)) = A_L g(x)$.
From this and \eqref{eq:S} we obtain
\begin{equation}
S(f_L(x)) = \det(A_L) S(x).
\label{eq:SfLx}
\end{equation}
Since $\det(A_L) > 0$ we can conclude that sign of $S$ (and thus also the sign of $\sin(\theta)$) is constant along orbits of $f_L$.
In particular, $\sin \left( \theta \left( f_L^{-i}(X) \right) \right)$ has the same sign for each $i \in \{ 0,1,\ldots,\ell-2 \}$.
Thus the angles $\theta \left( f_L^{-i}(X) \right)$ must be all less than $\pi$, all equal to $\pi$, or all greater than $\pi$.
But the path connecting $X, f^{-1}(X), \ldots, f^{-\ell}(X)$ (where $f = f_L$)
includes $X$ on the positive $x_2$-axis, $U$ on the negative $x_1$-axis, and $f_L^{-1}(U)$ on the negative $x_2$-axis,
and has no self-intersections (Lemma \ref{le:simplePolygon}).
Therefore the angles are all less than $\pi$.
By applying a similar argument to $f_R$
we conclude that all interior angles of $\Omega$ are less than $\pi$,
except possibly at the points $U$ and $Z$.

\myStep{2}{Convex subsets of $\Omega$.}\\
We now define two convex subsets of $\Omega$ (one in $x_2 \le 0$ and one in $x_2 \ge 0$).
Let $\Omega_{\rm upper}$ be the polygon formed by connecting the points 
\begin{equation}
U, f^{-(\ell-2)}(X), \ldots, f^{-1}(X), X, f(X),
\label{eq:OmegaUpper}
\end{equation}
in order, and from $f(X)$ back to $U$.
Let $\Omega_{\rm lower}$ be the polygon formed by connecting the points 
\begin{equation}
f(X), f^2(X), \ldots, f^{r-1}(X), Z, f(Y),
\label{eq:OmegaLower}
\end{equation}
in order, and from $f(Y)$ back to $f(X)$.
Since $U$ and $f(X)$ lie on $x_2 = 0$ while all other vertices of $\Omega_{\rm upper}$ lie in $x_2 \ge 0$,
the interior angles of $\Omega_{\rm upper}$ at $U$ and $f_R(x)$ are less than $\pi$.
Thus by the previous result $\Omega_{\rm upper}$ is convex.
For similar reasons, $\Omega_{\rm lower}$ is also convex.

\myStep{3}{Consequences of assumptions \eqref{eq:Yass}--\eqref{eq:Zass2}.}\\
All points between $U = (U_1,0)$, where $U_1 < 0$, and $f_R(X) = (\beta+1,0)$, where $\beta > 0$, belong to ${\rm int}(\Omega)$.
This includes $O = (0,0)$ and $f(O) = (1,0)$.
Also $f(Y) \in {\rm int}(\Omega)$ because \eqref{eq:Yass} implies that $f(Y)$ lies between $U$ and $f(O)$.
We now show $f(Z) \in {\rm int}(\Omega)$.
Assumptions \eqref{eq:Zass1} and \eqref{eq:Zass2} imply that either $Z = Y$ (in which case $f(Z) \in {\rm int}(\Omega)$ is immediate)
or $Z$ belongs to the interior of the quadrilateral $\mathcal{Q} = f^{-1}(U) V U O$.
Each vertex of $\mathcal{Q}$ belongs to $\Pi_L$ where $f = f_L$ is affine,
thus $f(\mathcal{Q})$ is the quadrilateral $U f(V) f(U) f(O)$.
Each vertex of $f(\mathcal{Q})$ belongs to $\Omega_{\rm upper}$, which is convex,
thus $f(\mathcal{Q}) \subset \Omega_{\rm upper}$.
Since $Z \in {\rm int}(\mathcal{Q})$ we have that $f(Z) \in {\rm int}(\Omega_{\rm upper}) \subset {\rm int}(\Omega)$.

\myStep{4}{Forward invariance of $\Omega$.}\\
Write $\Omega = \Omega_L \cup \Omega_R$ where
\begin{align}
\Omega_L &= \Omega \cap \Pi_L \,, \nonumber \\
\Omega_R &= \Omega \cap \Pi_R \,. \nonumber
\end{align}
Observe $f(\Omega) = f_L(\Omega_L) \cup f_R(\Omega_R)$.
Since $f_L$ is affine, $f_L(\Omega_L)$ is a polygon.
Evidently its vertices all belong to $\Omega_{\rm upper}$.
Since $\Omega_{\rm upper}$ is convex, $f_L(\Omega_L) \subset \Omega_{\rm upper} \subset \Omega$.
By a similar argument, $f_R(\Omega_R) \subset \Omega_{\rm lower} \subset \Omega$,
thus $\Omega$ is forward invariant.

\myStep{5}{Define $\Omega_{\rm trap}$.}\\
Let $P^{(1)}, \ldots, P^{(\ell+r)}$ denote the points \eqref{eq:OmegaVertices} in order.
That is, $P^{(1)} = U$ around to $P^{(\ell + r)} = Z$.
For each $j \in \{ 1,\ldots,\ell+r \}$ define
\begin{equation}
P^{(j)}_\ee = \left( 1 - \frac{\ee^j}{\| P^{(j)} \|} \right) P^{(j)},
\label{eq:Pjee}
\end{equation}
and assume $\ee$ is small enough that $1 - \frac{\ee^j}{\| P^{(j)} \|} > 0$ for each $j$.
Each $P^{(j)}_\ee$ is the result of moving from $P^{(j)}$ a distance $\ee^j$ towards $O$, see Fig.~\ref{fig:RLpSchem_n}.
Let $\Omega_{\rm trap}$ be the polygon with vertices $P^{(j)}_\ee$ (connected in the same order as for $\Omega$).
Immediately we have $d_H(\Omega,\Omega_{\rm trap}) \le \ee$.
Also $\Omega_{\rm trap} \subset \Omega$ because $\left[ P^{(j)}, O \right] \subset \Omega$ for each $j$.

\begin{figure}[b!]
\begin{center}
\includegraphics[height=7.5cm]{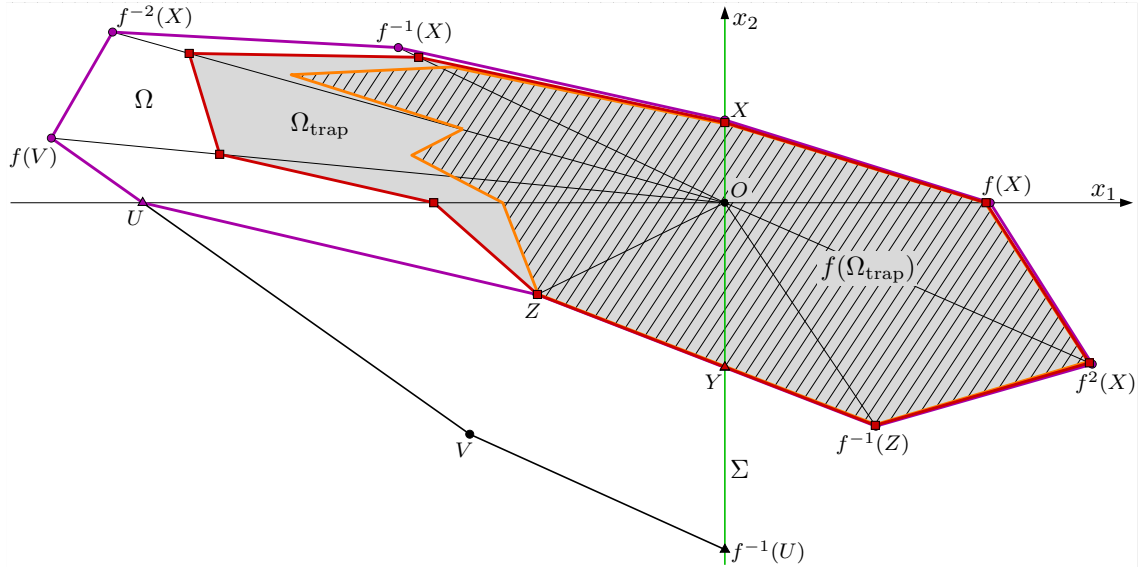}
\caption{
A plot of $\Omega$ (unshaded), $\Omega_{\rm trap}$ (shaded), and $f(\Omega_{\rm trap})$ (striped).
Here $r = 4$ and $\ell = 5$ as in Fig.~\ref{fig:RLpSchem_c}.
The set $\Omega_{\rm trap}$ is plotted by using $\ee = 0.5$ in \eqref{eq:Pjee}.
\label{fig:RLpSchem_n}
}
\end{center}
\end{figure}

\myStep{6}{Convex subsets of $\Omega_{\rm trap}$.}\\
Analogous to $\Omega_{\rm upper}$ and $\Omega_{\rm lower}$,
let $\Omega_{\rm trap,upper} \subset \Omega_{\rm trap}$ be the polygon in $x_2 \ge 0$ formed by connecting the points 
\begin{equation}
P^{(1)}_\ee, P^{(2)}_\ee, \ldots, P^{(\ell+1)}_\ee,
\label{eq:OmegaTrapUpper}
\end{equation}
in order, and from $P^{(\ell+1)}_\ee$ back to $P^{(1)}_\ee$.
Notice $P^{(1)}_\ee$ and $P^{(\ell+1)}_\ee$ lie on $x_2 = 0$.
Also let $Y_\ee \in \Sigma$ denote the intersection of
$\big[ P^{(\ell+r)}_\ee, P^{(\ell+r-1)}_\ee \big]$ with $\Sigma$ and
let $\Omega_{\rm trap,lower} \subset \Omega_{\rm trap}$ be the polygon in $x_2 \le 0$ formed by connecting the points 
\begin{equation}
P^{(\ell+1)}_\ee, P^{(\ell+2)}_\ee, \ldots, P^{(\ell+r)}_\ee, f(Y_\ee),
\label{eq:OmegaTrapLower}
\end{equation}
in order, and from $f(Y_\ee)$ back to $P^{(\ell+1)}_\ee$.
Notice $f(Y_\ee)$ lies on $x_2 = 0$ close to $f(Y)$.
Each interior angle of $\Omega_{\rm trap,upper}$ and $\Omega_{\rm trap,lower}$
is at most an order-$\ee$ perturbation of the corresponding interior angle of $\Omega_{\rm upper}$ or $\Omega_{\rm lower}$.
All interior angles of $\Omega_{\rm upper}$ and $\Omega_{\rm lower}$ are less than $\pi$,
so the same is true for $\Omega_{\rm trap,upper}$ and $\Omega_{\rm trap,lower}$
assuming $\ee$ is sufficiently small.
That is, $\Omega_{\rm trap,upper}$ and $\Omega_{\rm trap,lower}$ are convex.

\myStep{7}{The set $\Omega_{\rm trap}$ is a trapping region.}\\
Write $\Omega_{\rm trap} = \Omega_{{\rm trap},L} \cup \Omega_{{\rm trap},R}$ where
\begin{align}
\Omega_{{\rm trap},L} &= \Omega_{\rm trap} \cap \Pi_L \,, \nonumber \\
\Omega_{{\rm trap},R} &= \Omega_{\rm trap} \cap \Pi_R \,, \nonumber
\end{align}
and observe $f(\Omega_{\rm trap}) = f_L(\Omega_{{\rm trap},L}) \cup f_R(\Omega_{{\rm trap},R})$.

The vertices of $\Omega_{{\rm trap},L}$ are $Y_\ee$, $P^{(\ell+r)}_\ee$, and $P^{(j)}_\ee$ for $j = 1,\ldots,\ell$.
We now show that, if $\ee$ is sufficiently small, then these vertices all map under $f = f_L$ to either ${\rm int}(\Omega_{\rm trap,upper})$
or to a point on $x_2 = 0$ in the open line segment $I = \big( P^{(1)}_\ee, P^{(\ell+1)}_\ee \big)$.
Since $\Omega_{\rm trap,upper}$ is convex 
this implies $f_L(\Omega_{{\rm trap},L}) \subset {\rm int}(\Omega_{\rm trap,upper}) \cup I$.
Consequently $f_L(\Omega_{{\rm trap},L}) \subset {\rm int}(\Omega_{\rm trap})$
because $I \subset {\rm int}(\Omega_{\rm trap})$.

Certainly $f(Y_\ee), f \big( P^{(\ell)}_\ee \big) \in I$, assuming $\ee$ is sufficiently small.
If $Z \ne Y$ then $f \big( P^{(\ell+r)}_\ee \big) \in {\rm int}(\Omega_{\rm trap,upper})$,
assuming $\ee$ is sufficiently small, because $f(Z) \in {\rm int}(\Omega)$.
If $Z = Y$ then $f \big( P^{(\ell+r)}_\ee \big) \in I$.
Now choose any $j = 1,\ldots,\ell-1$.
By definition, $P^{(j)}_\ee = (1-s) P^{(j)} + s O$, with $s = \frac{\ee^j}{\| P^{(j)} \|}$.
In $\Pi_L$, $f = f_L$ is affine, so we have
\begin{equation}
f \left( P^{(j)}_\ee \right) = \begin{cases}
(1-s) f(U) + s f(O), & j = 1, \\
(1-s) P^{(j+1)} + s f(O), & j = 2,\ldots,\ell-1,
\end{cases}
\label{eq:fLPjee}
\end{equation}
Therefore, for $j \ne 1$, $f \big( P^{(j)}_\ee \big)$ is the result of moving from $P^{(j+1)}$ a distance
$\frac{\ee^j \| P^{(j+1)} - f(O) \|}{\| P^{(j)} \|}$ towards $f(O)$.
Thus $f \big( P^{(j)}_\ee \big)$ belongs to the triangle $P^{(j+1)} P^{(j+2)} O$, assuming $\ee$ is sufficiently small,
because $P^{(j+1)}$ and $P^{(j+2)}$ lie in $x_2 \ge 0$
with $P^{(j+2)}$ located clockwise (with respect to $O$) from $P^{(j+1)}$.
This is true in the case $j = 1$ also.
The distance from $f \big( P^{(j)}_\ee \big)$ to
$\big[ P^{(j+1)}, P^{(j+2)} \big]$ is proportional to $\ee^j$,
while $d_H \big( \big[ P^{(j+1)}_\ee, P^{(j+2)}_\ee \big], \big[ P^{(j+1)}, P^{(j+2)} \big] \big)$
is proportional to $\ee^{j+1}$.
Therefore, assuming $\ee$ is sufficiently small,
$f \big( P^{(j)}_\ee \big)$ lies in the triangle $P^{(j+1)}_\ee P^{(j+2)}_\ee O$
and not on the line segment $\big[ P^{(j+1)}_\ee, P^{(j+2)}_\ee \big]$.
Thus $f \big( P^{(j)}_\ee \big) \in {\rm int}(\Omega_{\rm trap,upper})$
and this completes our demonstration that $f_L(\Omega_{{\rm trap},L}) \subset {\rm int}(\Omega_{\rm trap})$.

From similar arguments it follows that
$f_R \left( \Omega_{{\rm trap},R} \right) \subset {\rm int}(\Omega_{\rm trap,lower}) \cup I$,
assuming $\ee$ is sufficiently small,
and so $f_R(\Omega_{{\rm trap},R}) \subset {\rm int}(\Omega_{\rm trap})$.
Hence $\Omega_{\rm trap}$ is a trapping region for $f$.
\end{proof}

\begin{proof}[Proof of Proposition \ref{pr:goFromOmegaGen}]
We first show $\Omega \subset \bigcup_{i=-\ell}^0 f^i(\Omega_{\rm rec})$.
Choose any $x \in \Omega$.
If $x_1 > 0$ then $x \in \Omega_{\rm rec}$.
If $x_1 \le 0$ then $x \in D_i$, for some $i \in \{ 1,\ldots,\ell \}$.
The upper bound $i = \ell$ is a consequence of \eqref{eq:Yass}--\eqref{eq:Zass2}
because $V$ lies above $f_L^{-\ell}(\Sigma)$ and $f^{-1}(U)$ lies on or above $f_L^{-\ell}(\Sigma)$.
Thus by Proposition \ref{pr:partition}, $f^i(x)_1 > 0$, and so $f^i(x) \in \Omega_{\rm rec}$ because $\Omega$ is forward invariant.

Now choose any $y \in \Omega_{\rm rec}$.
If $f(y) \in \Omega_{\rm rec}$, let $n = 1$ and observe that the first symbol of any $\cS \in \Gamma(y)$ is $R$, which belongs to $\bW$.
So now suppose $f(y) \notin \Omega_{\rm rec}$.
Also suppose $f(Y)_1 < 0$, so then $f(y)$ belongs to the quadrilateral $Y Z f(Y) O$
(if instead $f(Y)_1 \ge 0$ then the following arguments can be applied to the part of $Y Z f(Y) O$ that belongs to $\Pi_L$).
We have $\chi_L(f(y)) \le \max \left[ \chi_L(Y), \chi_L(Z), \chi_L(f(Y)), \chi_L(O) \right]$
(this follows from the linear ordering of the regions $D_p$).
But $\chi_L(f(Y)) = \chi_L(Y) - 1$ and $\chi_L(O) = 1$, thus $\chi_L(f(y)) \le p_{\rm max}$.

Let $n = \chi_L(f(y)) + 1$.
Then $f^n(y)_1 > 0$ and so $f^n(y) \in \Omega_{\rm rec}$ because $\Omega$ is forward invariant.
Also $f^j(y)_1 \le 0$ for all $j = 1,\ldots,n-1$ with $f^j(y)_1 = 0$ only possible for $j = 1$ and $j = n-1$.
In summary, $y_1 > 0$, $f(y)_1 \le 0$, $f^j(y)_1 < 0$ for all $j = 2,\ldots,n-2$, $f^{n-1}(y)_1 \le 0$, and $f^n(y) \in \Omega_{\rm rec}$.
Thus there are four possibilities for the first $n$ symbols of $\cS \in \Gamma(y)$:
$R L^{n-1}$, $R R L^{n-2}$, $R L^{n-2} R$, and $R R L^{n-3} R$
(the last possibility can only arise if $f(y)_1 = 0$ and $f^{n-1}(y)_1 = 0$).
All four words can be expressed as a concatenation of words in $\bW$ (because $n - 1 \le p_{\rm max}$).
Thus $\Omega_{\rm rec}$ is $\bW$-recurrent.
\end{proof}

\section{An algorithm for detecting a chaotic attractor}
\label{sec:algorithm}
\setcounter{equation}{0}

In the previous two sections we obtained sufficient conditions
for the assumptions of Theorem \ref{th:main} to hold with a trapping region for the 2d BCNF \eqref{eq:bcnf}.
In \S\ref{sub:algorithm} we summarise these conditions and state Algorithm \ref{al:theAlgorithm} (in pseudo-code)
for testing their validity.
In \S\ref{sub:comments} we further discuss the application of the algorithm to the
slice of parameter space shown in Fig.~\ref{fig:RLpNumerics_0}.

\subsection{Statement and proof of the algorithm}
\label{sub:algorithm}

The polygon $\Omega$ constructed in \S\ref{sub:Omega}
typically satisfies \eqref{eq:Yass}--\eqref{eq:Zass2}
for some interval of $\beta$-values (where $X = (0,\beta)$).
Within this interval, smaller values of $\beta$ tend to correspond to smaller values of $\chi_L(Y)$ and $\chi_L(Z)$
and so produce a smaller value for $p_{\rm max}$, \eqref{eq:pMax2}.
Smaller values of $p_{\rm max}$ are more favourable for the cone $C_J$ \eqref{eq:CJ} to be well-defined, invariant, and expanding.
This is because with a smaller value of $p_{\rm max}$ there are less matrices in $\bM$
and therefore fewer inequalities that need to be satisfied.

For these reasons we search for a suitable value of $\beta$
by iteratively increasing its value in steps of size $\beta_{\rm step}$ from $\beta_{\rm min}$ up to (at most) $\beta_{\rm max}$.
To produce Fig.~\ref{fig:RLpNumerics_0} we used
\begin{align}
\beta_{\rm step} &= 0.01, &
\beta_{\rm min} &= 0.01, &
\beta_{\rm max} &= 5.
\end{align}

For a given value of $\beta$ there are five groups of conditions that need to be checked.
These are labelled (C1)--(C5) in Algorithm \ref{al:theAlgorithm} below.
First we require $\Omega$ to be well-defined.
This is established by showing that $r$ and $\ell$ of Definition \ref{df:rl} exist.
To produce Fig.~\ref{fig:RLpNumerics_0} this was implemented by
iterating $X$ backwards and forwards up to maximum allowed values
\begin{align}
r_{\rm max} &= 15, &
\ell_{\rm max} &= 15.
\end{align}
Second we check conditions \eqref{eq:Yass}--\eqref{eq:Zass2}.
If these are satisfied then $\Omega$ is forward invariant
and in Algorithm \ref{al:theAlgorithm} this fixes the value of $\beta$.
We then evaluate $p_{\rm max}$ by iterating $Y$ and $Z$ under $f$, \eqref{eq:pMax2}.
The remaining three conditions are that the cone $C_J$ is well-defined,
that $C_J$ is invariant,
and that $C_J$ is expanding.
The computations involved in the last two steps are elementary because $G_j(m)$ and $H_j(m)$ are
polynomials of degree two or less.
Algorithm \ref{al:theAlgorithm} registers its success or failure
by the termination value of the Boolean variable $\chi_{\tt chaos}$.

\newpage  
\begin{algorithm}~\\
{\tt
set $\chi_{\text{chaos}} = {\color{codeRed} \text{false}}$\\
set $\beta = \beta_{\rm min} > 0$\\
{\color{codeGreen} While} $\chi_{\text{chaos}} = {\color{codeRed} \text{false}}$ and $\beta \le \beta_{\rm max}$\\
\makebox[0pt][l]{\circled{\rm C1}}\codeIndent {\color{codeGreen} If} $r$ or $\ell$ do not exist\\
\codeIndent \codeIndent set $\beta = \beta + \beta_{\rm step}$\\
\codeIndent {\color{codeGreen} else}\\
\makebox[0pt][l]{\circled{\rm C2}}\codeIndent \codeIndent {\color{codeGreen} If} any of \eqref{eq:Yass}{\rm --}\eqref{eq:Zass2} are false\\
\codeIndent \codeIndent \codeIndent set $\beta = \beta + \beta_{\rm step}$\\
\codeIndent \codeIndent {\color{codeGreen} else}\\
\codeIndent \codeIndent \codeIndent set $\chi_{\text{chaos}} = {\color{codeRed} \text{true}}$\\
\codeIndent \codeIndent {\color{codeGreen} end}\\
\codeIndent {\color{codeGreen} end}\\
{\color{codeGreen} end}\\
{\color{codeGreen} If} $\chi_{\text{chaos}} = {\color{codeRed} \text{true}}$\\
\codeIndent Evaluate $p_{\rm max}$ \eqref{eq:pMax2}.\\
\makebox[0pt][l]{\circled{\rm C3}}\codeIndent {\color{codeGreen} If} \eqref{eq:Mjassumptions} is false for some $M^{(j)} = A_L^{j-1} A_R$ with $j \in \{ 1,\ldots,p_{\rm max}+1 \}$\\
\codeIndent \codeIndent set $\chi_{\text{chaos}} = {\color{codeRed} \text{false}}$\\
\codeIndent {\color{codeGreen} else}\\
\codeIndent \codeIndent Evaluate $m^{(j)}_{\rm stab}$ and $m^{(j)}_{\rm unstab}$ for each $j$.\\
\codeIndent \codeIndent Evaluate $m_{\rm stab,min}$ and $m_{\rm stab,max}$.\\
\makebox[0pt][l]{\circled{\rm C4}}\codeIndent \codeIndent {\color{codeGreen} If}
$m_{\rm stab,min} \le m^{(j)}_{\rm unstab} \le m_{\rm stab,max}$ for some $j \in \{ 1,\ldots,p_{\rm max}+1 \}$\\
\codeIndent \codeIndent \codeIndent set $\chi_{\text{chaos}} = {\color{codeRed} \text{false}}$\\
\codeIndent \codeIndent {\color{codeGreen} else}\\
\makebox[0pt][l]{\circled{\rm C5}}\codeIndent \codeIndent \codeIndent {\color{codeGreen} If}, for some $j \in \{ 1,\ldots,p_{\rm max}+1 \}$, $H_j$ does not have two distinct\\
\codeIndent \codeIndent \codeIndent \codeIndent \codeIndent real roots or two of \eqref{eq:inequality1}{\rm --}\eqref{eq:inequality3} are false (or, if\\
\codeIndent \codeIndent \codeIndent \codeIndent \codeIndent $b_j^2 + d_j^2 = 1$, the condition in Remark \ref{re:oneRoot} is false)\\
\codeIndent \codeIndent \codeIndent \codeIndent set $\chi_{\text{chaos}} = {\color{codeRed} \text{false}}$\\
\codeIndent \codeIndent \codeIndent {\color{codeGreen} end}\\
\codeIndent \codeIndent {\color{codeGreen} end}\\
\codeIndent {\color{codeGreen} end}\\
{\color{codeGreen} end}
}
\label{al:theAlgorithm}
\end{algorithm}

The theorem below assumes calculations are done exactly.
For Fig.~\ref{fig:RLpNumerics_0} calculations were performed with rounding at $16$ digits.

\begin{theorem}
Let $f$ be a map of the form \eqref{eq:bcnf} with $\delta_L, \delta_R > 0$.
If Algorithm \ref{al:theAlgorithm} outputs $\chi_{\tt chaos} = {\tt true}$
then $f$ has an attractor with a positive Lyapunov exponent.
\label{th:algorithm}
\end{theorem}

\begin{proof}
Suppose Algorithm \ref{al:theAlgorithm} outputs $\chi_{\tt chaos} = {\tt true}$.
Then (C1)--(C5) all hold for some fixed $\beta > 0$.
Since (C1) holds, $\Omega$ is well-defined by Lemma \ref{le:simplePolygon}.
Since (C2) holds, $f$ has a trapping region $\Omega_{\rm trap} \subset \Omega$ by Proposition \ref{pr:forwardInvariant}.
Thus $f$ has an attractor $\Lambda \subset \Omega_{\rm trap}$.
Let $\bW$ be given by \eqref{eq:bW}
and $\Omega_{\rm rec}$ be given by \eqref{eq:OmegaGen}.
Then, by Proposition \ref{pr:goFromOmegaGen},
$\Lambda \subset \bigcup_{i=-\infty}^\infty f^i(\Omega_{\rm rec})$
and $\bW$ generates $\Gamma(y)$ for all $y \in \Omega_{\rm rec}$.

Since (C3) holds, the cone $C_J$ is well-defined.
Since (C4) holds, $C_J$ is forward invariant under
$\bM = \left\{ \pM(\cW) \,\middle|\, \cW \in \bW \right\}$ by Proposition \ref{pr:CJ}.
Since (C5) holds, $C_J$ is also expanding under $\bM$ by Propositions \ref{pr:CJ} and \ref{pr:threeInequalities}.
Then by Theorem \ref{th:main}
for all $x \in \Lambda$ there exists $v \in T \mathbb{R}^2$ such that $\lambda(x,v) > 0$.
\end{proof}

\subsection{Comments on the results of Algorithm \ref{al:theAlgorithm}}
\label{sub:comments}

As mentioned in \S\ref{sec:bcnf}, for \eqref{eq:bcnf} with $\delta_L = \delta_R = 0.3$,
Algorithm \ref{al:theAlgorithm} outputs $\chi_{\tt chaos} = {\tt true}$
throughout the red regions of Fig.~\ref{fig:RLpNumerics_0}.
Here we examine three sample parameter combinations in detail.
For each of the three black dots in Fig.~\ref{fig:RLpNumerics_0},
the polygon $\Omega$ is well-defined and forward invariant.
Figs.~\ref{fig:RLpNumerics_ab}a--\ref{fig:RLpNumerics_gh}a
show $\Omega$ using the value of $\beta > 0$ generated by Algorithm \ref{al:theAlgorithm}.

\begin{figure}[b!]
\begin{center}
\includegraphics[height=6cm]{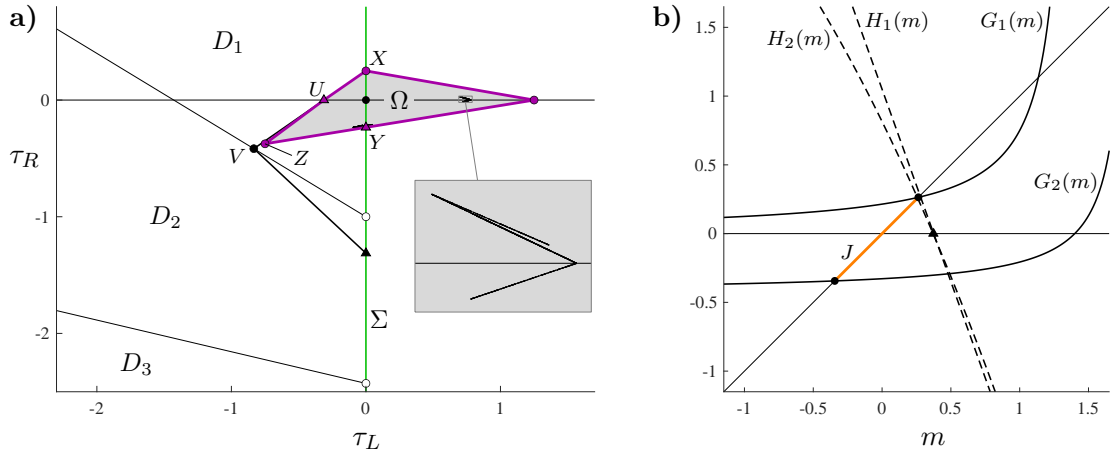}
\caption{
Constructive elements produced by Algorithm \ref{al:theAlgorithm}
for \eqref{eq:bcnf} with \eqref{eq:dLdR} and $(\tau_L,\tau_R) = (0.7,-1.4)$.
Here the algorithm obtains $\beta = 0.25$ and returns $\chi_{\tt chaos} = {\tt true}$. 
Panel (a) shows the forward invariant region $\Omega$ (see Fig.~\ref{fig:RLpSchem_c}),
the regions $D_p$ (see Fig.~\ref{fig:RLpSchem_b}),
and a numerically computed attractor.
Panel (b) shows the slope maps $G_j$ \eqref{eq:Gj} and the functions $H_j$ \eqref{eq:Hj} for $j = 1,2$.
\label{fig:RLpNumerics_ab}
}
\end{center}
\end{figure}

In Fig.~\ref{fig:RLpNumerics_ab}a we have $Y, Z \in D_1$, so $p_{\rm max} = 1$ and $\bW = \{ R, RL \}$.
Fig.~\ref{fig:RLpNumerics_ab}b shows how (C5) is satisfied.
With $j = 1$ we have $\cW = R$,
so $H_1(m)$ is linear, see Remark \ref{re:oneRoot}, and $m_{\rm root} > m_{\rm stab,max}$.
With $j = 2$ we have $\cW = RL$
with which $H_2(m)$ is quadratic and \eqref{eq:inequality1} and \eqref{eq:inequality3} are satisfied.
Numerical simulations suggest that at these parameter values $f$ has a unique two-piece chaotic attractor
with one piece intersecting $f(\Sigma)$ (as shown in the magnification of Fig.~\ref{fig:RLpNumerics_ab}a),
and its image intersecting $\Sigma$.

\begin{figure}[b!]
\begin{center}
\includegraphics[height=6cm]{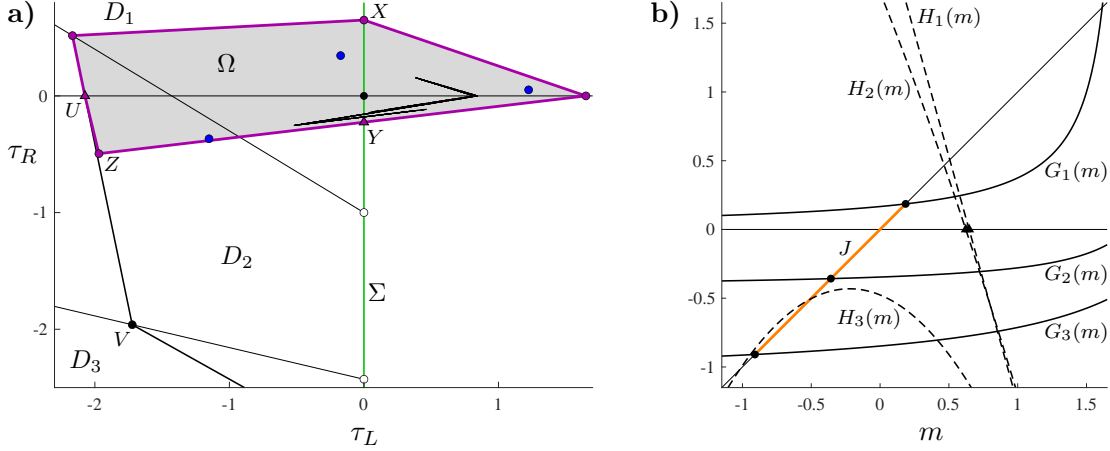}
\caption{
Constructive elements produced by Algorithm \ref{al:theAlgorithm}
for \eqref{eq:bcnf} with \eqref{eq:dLdR} and $(\tau_L,\tau_R) = (0.7,-1.8)$.
Here the algorithm obtains $\beta = 0.65$ and returns $\chi_{\tt chaos} = {\tt false}$. 
Panel (a) shows the forward invariant region $\Omega$, the regions $D_p$,
a numerically computed attractor, and a stable period-$3$ solution (blue circles).
Panel (b) shows $G_j$ and $H_j$ for $j = 1,2,3$.
\label{fig:RLpNumerics_cd}
}
\end{center}
\end{figure}

In Fig.~\ref{fig:RLpNumerics_cd}a we have $Y \in D_1$ and $Z \in D_2$, so $p_{\rm max} = 2$ and $\bW = \{ R, RL, RL^2 \}$.
Here Algorithm \ref{al:theAlgorithm} returns $\chi_{\tt chaos} = {\tt false}$ because (C5) is not satisfied.
This is because $H_3(m)$ has no real roots, and this is evident in Fig.~\ref{fig:RLpNumerics_cd}b.
Indeed at these parameter values $f$ has a stable period-$3$ solution corresponding to the word $RL^2$.
Nevertheless, $f$ does appear to have a chaotic attractor contained in $D_1 \cup \Pi_R$.
It may be possible to prove this attractor has a positive Lyapunov exponent
by constructing a trapping region in $D_1 \cup \Pi_R$.
For the parameter values of Fig.~\ref{fig:RLpNumerics_gh} we again have $p_{\rm max} = 2$ but
now $H_3(m)$ has real roots satisfying \eqref{eq:inequality1} and \eqref{eq:inequality3} and
Algorithm \ref{al:theAlgorithm} terminates with $\chi_{\tt chaos} = {\tt true}$.

\begin{figure}[b!]
\begin{center}
\includegraphics[height=6cm]{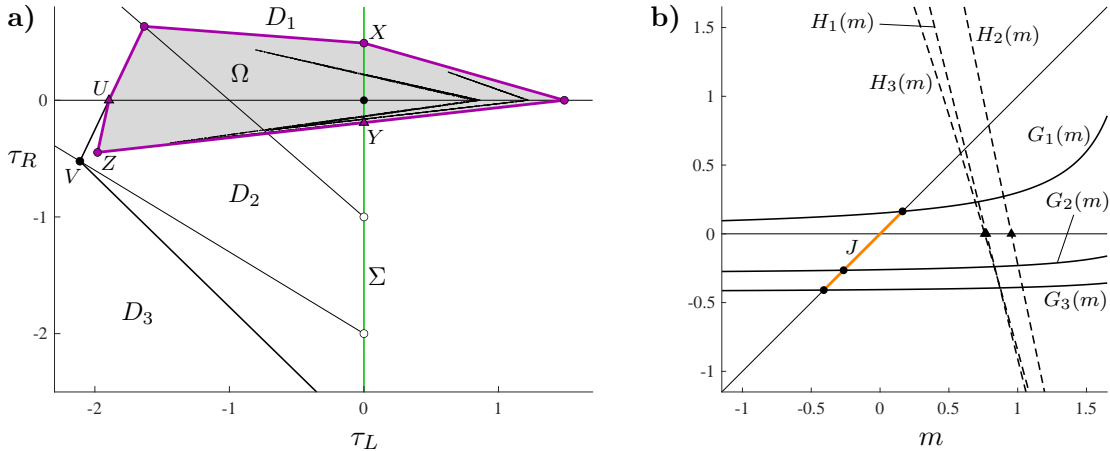}
\caption{
Constructive elements produced by Algorithm \ref{al:theAlgorithm}
for \eqref{eq:bcnf} with \eqref{eq:dLdR} and $(\tau_L,\tau_R) = (1,-2)$.
Here the algorithm obtains $\beta = 0.49$ and returns $\chi_{\tt chaos} = {\tt true}$. 
Panel (a) shows the forward invariant region $\Omega$, the regions $D_p$,
and a numerically computed attractor (shown also in Fig.~\ref{fig:RLpNumerics_ij}b).
Panel (b) shows $G_j$ and $H_j$ for $j = 1,2,3$.
\label{fig:RLpNumerics_gh}
}
\end{center}
\end{figure}

We now discuss the region boundaries in Fig.~\ref{fig:RLpNumerics_0} labelled $B_1$ to $B_5$.
Boundary $B_1$ is the horizontal line $\tau_R = -\delta_R - 1$.
Below $B_1$, and for $\tau_L > 0.7$ approximately, Algorithm \ref{al:theAlgorithm} terminates with $\chi_{\tt chaos} = {\tt true}$.
On $B_1$ (C5) is not satisfied because $A_R$ has an eigenvalue of $-1$
so $H_1(m) = 0$ at $m = m^{(1)}_{\rm stab}$.
Indeed above $B_1$ the map $f$ has an asymptotically stable fixed point in $x_1 > 0$.
In this way Algorithm \ref{al:theAlgorithm} detects a true bifurcation boundary between chaotic and non-chaotic dynamics.

On $B_2$ (C5) is not satisfied because $H_2(m) = 0$ at $m = m^{(1)}_{\rm stab}$.
Thus to the left of $B_2$, Algorithm \ref{al:theAlgorithm} terminates with $\chi_{\tt chaos} = {\tt false}$
because some $v \in C_J$ do not expand when multiplied by $M = A_L A_R$.
Nevertheless numerical results suggest $f$ has a chaotic attractor here.
It may be possible to prove this by using a different word set $\bW$.

Boundary $B_3$ is the upper boundary of the blue region in which
there exists a stable period-$3$ solution of period $n=3$ (corresponding to the word $RL^2$, see Fig.~\ref{fig:RLpNumerics_cd}a).
On this boundary the periodic solution is destroyed in a border-collision bifurcation
by having one of its points collide with $\Sigma$.
Algorithm \ref{al:theAlgorithm} does not detect this boundary exactly
as evident in Fig.~\ref{fig:RLpNumerics_0} by the presence of white pixels immediately above $B_3$.
At these pixels Algorithm \ref{al:theAlgorithm} obtains $p_{\rm max} = 2$ with which
(C5) is not satisfied because $H_3(m)$ has no real roots.
In nearby red pixels Algorithm \ref{al:theAlgorithm} obtains $p_{\rm max} = 1$ with which
the behaviour of $H_3(m)$ is irrelevant.
The number of white pixels appears to tend to zero in the limit $\beta_{\rm step} \to 0$ because
the size of the interval of $\beta$-values for which $\Omega$ is forward invariant with $p_{\rm max} = 1$
vanishes as we approach $B_3$ from above.
In a similar way as we approach the homoclinic bifurcation HC from above
the size of the interval of $\beta$-values for which (C1) and (C2) are satisfied approaches zero
(in \cite{GlSi19} a different approach was used to construct a trapping region).

On $B_4$ the period-three solution loses stability by attaining an eigenvalue of $-1$.
For $\tau_R < -2.6$, approximately, Algorithm \ref{al:theAlgorithm} detects this boundary exactly.
On $B_4$ (C5) is not satisfied because $H_3(m) = 0$ at $m = m^{(3)}_{\rm stab}$.
That is, $\| M v \| = \| v \|$ for the eigenvector
$v = \big( 1, m^{(3)}_{\rm stab} \big)$ of $M = A_L^2 A_R$ corresponding to the eigenvalue $-1$.
Finally, boundary $B_5$ is analogous to boundary $B_2$.
On $B_5$ we have $H_3(m) = 0$ at $m = m^{(1)}_{\rm stab}$.

\section{Discussion}
\label{sec:conc}
\setcounter{equation}{0}

We have presented a general method by which
one can prove, possibly with computer assistance, that a piecewise-linear map has a chaotic attractor.
We applied the method to the 2d BCNF and found a chaotic attractor throughout a parameter regime
that, unlike the logistic family for example, does not contain periodic windows.
Such {\em robust chaos} is typical for piecewise-linear maps
and for this reason piecewise-linear maps are desirable in applications that use chaos such as chaos-based cryptography \cite{KoLi11}.

In our implementation we considered only one approach for the construction of $\Omega_{\rm trap}$
and only word sets of the form \eqref{eq:bW}.
There is considerable room to generalise these, 
such as by defining $\Omega_{\rm trap}$ be to the union of a polygon and its images under $f$ \cite{Si20b}.

A major next step would be the application of this method to families of higher-dimensional maps,
such the $N$-dimensional border-collision normal form.
Results of this nature have already been achieved in \cite{Gl15b,GlJe15}.
To construct a trapping region and a cone it may be helpful to work with convex polytopes \cite{AtLa14}.
It would also be useful to obtain a converse to Theorem \ref{th:main}:
if $f$ has a topological attractor with a positive Lyapunov exponent,
must some $\Omega_{\rm trap}$, $\bW$, and $C$ (satisfying the required properties) exist?

\section*{Acknowledgements}

This work was supported by Marsden Fund contract MAU1809, managed by Royal Society Te Ap\={a}rangi.

\appendix
\section{The significance of the 2d BCNF}
\label{app:coordChange}

Let $f$ be a continuous map on $\mathbb{R}^2$ that is affine on each side of $\Sigma = \left\{ x \,\big|\, x_1 = 0 \right\}$.
Then $f$ has the form
\begin{equation}
f(x) = \begin{cases}
\begin{bmatrix} a_L & b \\ c_L & d \end{bmatrix}
\begin{bmatrix} x_1 \\ x_2 \end{bmatrix} +
\begin{bmatrix} p \\ q \end{bmatrix}, & x_1 \le 0, \\
\begin{bmatrix} a_R & b \\ c_R & d \end{bmatrix}
\begin{bmatrix} x_1 \\ x_2 \end{bmatrix} +
\begin{bmatrix} p \\ q \end{bmatrix}, & x_1 \ge 0,
\end{cases}
\label{eq:f2dgen}
\end{equation}
for some $a_L, a_R, b, c_L, c_R, d, p, q \in \mathbb{R}$.
It is a simple exercise to show that $f(\Sigma)$ intersects $\Sigma$ at a unique point if and only if $b \ne 0$.
Moreover, if $b \ne 0$ then this point is not a fixed point of \eqref{eq:f2dgen} if and only if $\xi = (1-d) p + b q \ne 0$.

Now suppose $b \ne 0$ and $\xi \ne 0$.
Then the coordinate change
\begin{equation}
\tilde{x} = \frac{1}{\xi} \left( \begin{bmatrix} 1 & 0 \\ -d & b \end{bmatrix} x + \begin{bmatrix} 0 \\ d p - b q \end{bmatrix} \right),
\label{eq:f2dcoordChange}
\end{equation}
is well-defined and invertible.
Also notice it leaves $\Sigma$ unchanged.
By directly applying \eqref{eq:f2dcoordChange} to \eqref{eq:f2dgen} we find that if $\xi > 0$ then $f$
is transformed to \eqref{eq:bcnf} with $\tilde{x}$ in place of $x$ and
$\tau_L = a_L + d$, $\delta_L = a_L d - b c_L$, $\tau_R = a_R + d$, and $\delta_R = a_R d - b c_R$. 
If instead $\xi < 0$ then
$\tau_L = a_R + d$, $\delta_L = a_R d - b c_R$, $\tau_R = a_L + d$, and $\delta_R = a_L d - b c_L$.

{\footnotesize

}

\end{document}